\documentclass[11pt, reqnom,twoside]{amsart}

\usepackage{array,amssymb,amsxtra,latexsym,graphicx,psfrag,epsfig,ifthen,mathtools,mathrsfs}
\usepackage[all]{xy}
\SelectTips {cm}{}
\usepackage{bbm,bm}
\usepackage{rotating}
\usepackage{float}
\usepackage{caption}
\usepackage{subcaption}
\usepackage{comment}

\usepackage{xcolor}

\definecolor{mygreen}{rgb}{0,.4,0}
\definecolor{myblue}{rgb}{0,0,.5}
\definecolor{mymagenta}{cmyk}{0,.6,0,0}

\usepackage[colorlinks=true,final,hyperindex, pdfstartview=FitV,,, linkcolor=violet, citecolor=teal, urlcolor=cyan]{hyperref}

\usepackage[nameinlink,capitalize,noabbrev]{cleveref}

\usepackage[pro]{fontawesome5}

\usepackage[newitem,newenum,neverdecrease]{paralist}
\makeatletter
\if@plflushright
  
\else
  
\fi
\renewcommand{\@asparaenum@}{%
  \expandafter\list\csname label\@enumctr\endcsname{%
    \usecounter{\@enumctr}%
    \labelwidth\z@
    \labelsep.5em
    \leftmargin\z@
    \parsep\parskip
    \itemsep\z@
    \topsep\z@
    \partopsep\parskip
    \itemindent\parindent
    \advance\itemindent\labelsep
    \def\makelabel##1{\upshape ##1}}}
\makeatother
\usepackage{psfrag}
\usepackage{longtable}
\usepackage{shuffle}
\usepackage{fontawesome5}

\theoremstyle{plain}
\newtheorem{theorem}{Theorem}[section]
\newtheorem{proposition}[theorem]{Proposition}
\newtheorem{lemma}[theorem]{Lemma}
\newtheorem{corollary}[theorem]{Corollary}
\newtheorem{conjecture}[theorem]{Conjecture}
\newtheorem*{theorem*}{Theorem}

\theoremstyle{definition}
\newtheorem{definition}[theorem]{Definition}
\newtheorem{notation}[theorem]{Notation}

\newtheorem{example}[theorem]{Example}
\newtheorem{remark}[theorem]{Remark}

\numberwithin{equation}{section}

\newcommand{\id}{\mathrm{id}}

\newcommand{\xyinc}{\ar@{^{(}->}}
\newcommand{\xyrinc}{\ar@{_{(}->}}
\newcommand{\xyonto}{\ar@{->>}}
\newcommand{\xytwo}{\ar@{<->}}

\newcommand{\B}[1]{\mathtt{#1}} 
\DeclareMathAlphabet{\mathpzc}{OT1}{pzc}{m}{it}

\newcommand{\calA}{\mathcal{A}}

\newcommand{\beq}{\begin{equation}}
\newcommand{\eeq}{\end{equation}}

\usepackage{xargs}                      
\usepackage{xcolor}  
\usepackage[colorinlistoftodos,textsize=tiny]{todonotes}
\newcommandx{\unsure}[2][1=]{\todo[linecolor=red,backgroundcolor=red!25,bordercolor=red,#1]{#2}}
\newcommandx{\change}[2][1=]{\todo[linecolor=blue,backgroundcolor=blue!25,bordercolor=blue,#1]{#2}}
\newcommandx{\info}[2][1=]{\todo[linecolor=teal,backgroundcolor=teal!25,bordercolor=teal,#1]{#2}}
\newcommandx{\improvement}[2][1=]{\todo[caption={Short note},linecolor=violet,backgroundcolor=violet!25,bordercolor=violet,size=\tiny,#1]{#2}}
\newcommandx{\thiswillnotshow}[2][1=]{\todo[disable,#1]{#2}}
\newcommand{\sk}{\operatorname{sk}}
\newcommand{\ST}{\mathsf{Sch}}
\newcommand{\T}{\mathcal{T}}
\newcommand{\A}{\mathcal{A}}
\newcommand{\NC}{\mathsf{NC}}
\newcommand\teli[1]{{\color{cyan}#1}}
\newcommand{\uno}{{\bf{1}}}
\newcommand{\E}{\mathcal{E}}
\newcommand{\PST}{\mathsf{PSch}}
\newcommand{\sS}{\mathsf{S}}
\newcommand{\tT}{\mathsf{T}}
\newcommand{\internal}{i}

\newcommand{\Abs}{\operatorname{Abs}}

\newenvironment{arb}{\begin{tikzpicture}[baseline,scale=0.5,level distance=7mm,level 1/.style={sibling distance=10mm},level 2/.style={sibling distance=5mm},level 3/.style={sibling distance=3mm},grow=down, font=\scriptsize]
\tikzstyle{ve}=[draw,circle,inner sep=1pt,fill] 
\tikzstyle{vv}=[draw,circle,inner sep=1pt] 
\tikzstyle{vee}=[minimum size=0pt ,inner sep=0pt]}{\end{tikzpicture}}
\newenvironment{arbb}{\begin{tikzpicture}[baseline,scale=0.5,level distance=7mm,level 1/.style={sibling distance=25mm},level 2/.style={sibling distance=5mm},level 3/.style={sibling distance=3mm},grow=down, font=\scriptsize]
\tikzstyle{ve}=[draw,circle,inner sep=1pt,fill] 
\tikzstyle{vv}=[draw,circle,inner sep=1pt] 
\tikzstyle{vee}=[minimum size=0pt ,inner sep=0pt]}{\end{tikzpicture}}

\newcommand{\rd}[1]{\node[ve,label=above:$#1$] {}}

\newcommand{\vb}[1]{node[ve,label=below:$#1$] {}}

\newenvironment{arbolito}{\begin{tikzpicture}[baseline={([yshift=-1.5ex]current bounding box.center)},scale=0.6,level distance=7mm,level 1/.style={sibling distance=10mm},level 2/.style={sibling distance=5mm},level 3/.style={sibling distance=3mm},grow=down, font=\scriptsize]
\tikzstyle{ve}=[draw,circle,inner sep=1pt,fill] 
\tikzstyle{vv}=[draw,circle,inner sep=1pt] 
\tikzstyle{vee}=[draw,circle,minimum size=0pt ,inner sep=0pt]}{\end{tikzpicture}}

\def\corollaa{\begin{arb}
\rd{(i;i_0)}
child{\vb{i_1}} child[dotted] child[dotted] child{\vb{i_k}};
\end{arb}}

\usepackage{forest}
\forestset{
  decor/.style = {
    label/.expanded = {[inner sep = 0.2ex, font=\unexpanded{\tiny}]right:{$#1$}}
  },
  root/.style = {minimum size = 0.1ex},
  decorated/.style = {
    for tree = {
      circle, fill, inner sep = 0.3ex, minimum size = 1.ex,
      grow' = south, l = 0, l sep = 1.2ex, s sep = 0.7em,
      fit = tight, parent anchor = center, child anchor = center,
      delay = {decor/.option = content, content =}
    }
  },
  default preamble = {decorated, root},
  begin draw/.code={\begin{tikzpicture}[baseline={([yshift=-0.5ex]current bounding box.center)}]
  \end{tikzpicture}},
}

\tikzset{
solid node/.style={circle,draw,inner sep=1.5,fill=black, minimum size=1pt},
hollow node/.style={circle,draw,inner sep=1.5,fill=white,  minimum size=1pt}
}

\begin{document}
\title[Schröder trees, antipode formulas and non-commutative probability]{Schr$\ddot{\text{O}}$der trees, antipode formulas and non-commutative probability}

\author[Adrián Celestino]{Adri\'an Celestino}
\address[A. Celestino, Y. Vargas]{Institut für Diskrete Mathematik, Technische Universität Graz, Steyrergasse 30, 8010 Graz,
Austria.}
\email{celestino@math.tugraz.at}
\urladdr{https://sites.google.com/view/adriancelestino/}

\author[Yannic Vargas]{Yannic Vargas}
\email{yvargaslozada@tugraz.at}
\urladdr{https://yannicmath.github.io/}
\date{} 

\maketitle

\begin{abstract}
We obtain a cancellation-free formula, represented in terms of Schröder trees, for the antipode in the double tensor Hopf algebra introduced by Ebrahimi-Fard and Patras. We apply the antipode formula in the context of non-commutative probability and recover cumulant-moment formulas as well as a new expression for Anshelevich's free Wick polynomials in terms of Schröder trees.
\end{abstract}


\setcounter{tocdepth}{3}

\section{Introduction}
\label{sec:introduction}

In recent years, there has been a growing interest in the exploration of non-commutative analogues to probability theory, now recognized as the field of \emph{non-commutative probability} (also known as \emph{quantum probability}). Dan Voiculescu, in the 1980s, introduced the notion of \emph{free independence}, aiming to solve the problem of isomorphism between von Neumann algebras generated by free groups. Since then, \textit{free probability theory} \cite{NSp, voiculescu1992free} has encountered many applications and links with random matrix theory, operator algebras, combinatorics, entropy and quantum information theory. Later, different notions of non-commutative independence have been considered, such as \textit{Boolean independence} \cite{SpW} and \textit{monotone independence} \cite{muraki2000monotonic}, each being rich enough to define a corresponding theory of non-commutative probability.

\

Roughly speaking, the idea in non-commutative probability is to look at random variables as abstract elements inside a unital algebra $\A$ with unit $1_\A$ and to replace the usual expectation with a linear functional $\varphi:\A\to\mathbb{C}$, with $\varphi(1_\A) = 1$, called the \textit{non-commutative expectation}. In the classical case, the independence of random variables provides a recipe to compute mixed moments: if two random variables $X$ and $Y$ are independent, then $\mathbb{E}(X^mY^n) = \mathbb{E}(X^m)\mathbb{E}(Y^n)$, for any $m,n\geq0$. In the non-commutative setting, the independence of non-commutative random variables $a, b\in\A$ can also be defined as a recipe to compute $\varphi(a^{m_1}b^{n_1}\cdots a^{m_k}b^{n_k})$ in terms of the sequences $\{\varphi(a^n)\}_{n\geq0}$ and $\{\varphi(b^n)\}_{n\geq0}$. 

\

One of the most important concepts in non-commutative probability is the notion of cumulants, which have proven to be a significant tool in the combinatorial study of non-commutative probability. Roland Speicher introduced in \cite{Spe94} the notion of \emph{free cumulants} and proved that, from a
combinatorial point of view, the transition from classical probability to free probability consists in replacing the lattice of all set partitions with the lattice of \textit{non-crossing partitions}. Moreover, the analogous of \textit{Boolean cumulants} \cite{SpW} and \emph{monotone cumulants} \cite{HS11} have been defined, using the lattices of interval and monotone partitions, respectively. On the other hand, Anshelevich in \cite{Ansh04} introduced the notion of \textit{free Wick polynomials}, which is the counterpart in free probability of the classical Wick polynomials \cite{avram1987noncentral}. More latterly, the works of \cite{AC22, biane2023combinatorics, JVNT} have shown connections between the combinatorics of non-commutative probability and \textit{Schröder trees}. In particular, in \cite{AC22,JVNT}, the authors provide an alternative description of the cumulant-moment formulas by using Schröder trees instead of non-crossing partitions.

\

The notions of \textit{bialgebras} and \textit{Hopf algebras} (see \cite[Chap. 3]{CaP22}) have been proven to be helpful to understand in an alternative way the algebraic and combinatorial aspects of non-commutative probability theory, see for instance \cite{ASvW1987, SchurmannBook, NicaMastnak}. More generally, the work of Saj-Nicole A. Joni and Gian-Carlo Rota in \cite{JoniRota} was one of the first to emphasize the role of Hopf algebras as a language and a box tool to address specific problems in combinatorics: the notions of ``assembling'' and ``disassembling'' families of combinatorial objects, such as permutations, non-crossing partitions and trees, can be understood in an algebraic setting under products and coproducts.

\

More recently, Ebrahimi-Fard and Patras, in a series of papers \cite{EFP:2015, EFP18, EFP20}, have developed a Hopf-algebraic and pre-Lie-theoretical framework for cumulants in non-commutative probability. In this framework, the free, Boolean and monotone cumulants are interpreted as linear functionals on a word Hopf algebras, and the relations between moments and the different types of cumulants can be explained analogously to the link between a Lie algebra and its corresponding Lie group. Ebrahimi-Fard and Patras' point of view has proven to be a powerful framework for understanding cumulants and their combinatorial relations in a unified way, as well as the different notions of additive convolutions originated by each type of independence. 

\ 

In more precise words, the Hopf algebra considered by Ebrahimi-Fard and Patras in \cite{EFP:2015} is based on the \emph{double tensor module on a vector space $V$}, that is the vector space $\tT(\tT_+(V))$ defined by
$$ \tT(\tT_+(V)) := \bigoplus_{n\geq0} \tT_+(V)^{\otimes n},\quad\mbox{ where }\quad \tT_+(V):= \bigoplus_{n>0} V^{\otimes n}.$$
The authors of \cite{EFP:2015} introduced a product and a coproduct which provide to $\tT(\tT_+(V))$ a structure of connected graded bialgebra. By the general theory of Hopf algebras, an antipode can be recursively defined on each graded component of $\tT(\tT_+(V))$ providing it automatically with a Hopf algebra structure, called the \textit{double tensor Hopf algebra on $V$}. The link with non-commutative probability arises by identifying cumulants as the solutions of certain linear fixed-point equations in the dual of $\tT(\tT_+(V))$ (\cref{thm:link}).

\ 

The main objective of the present manuscript is to address the problem of computing a cancellation-free combinatorial formula for the antipode in $\tT(\tT_+(V))$, expressed in terms of Schröder trees. Our motivation comes from a general method developed by Menous and Patras in \cite{Menous2018} to compute a forest-type formula for the iterated coproduct and the antipode of a specific class of commutative polynomial Hopf algebras. We apply this method to first compute the antipode in $\sS(\tT_+(V))$, the \textit{symmetric algebra on $\tT_+(V)$}, that is, the Hopf algebra of symmetric tensors over $\tT_+(V)$ which can be thought as the commutative analogue of $\tT(\tT_+(V))$. The method in \cite{Menous2018} provides a combinatorial formula for the iterations of the reduced coproduct in $\sS(\tT_+(V))$ based on certain decorated non-planar trees, where the non-planarity of the trees reflects the fact that the symmetric algebra is commutative. Our combinatorial analysis shows that the decorated trees of Menous and Patras can be replaced by Schröder trees whose internal vertices are labelled increasingly, respecting the poset structure. It turns out that Schröder trees effectively describe the actions of the iterated coproduct and the antipode in $\sS(\tT_+(V))$ (\cref{thm:SymAntipode}).

\ 

In the present manuscript, we provide a combinatorial analysis that shows how cancellations are performed in the non-commutative case in order to arrive at the combinatorial formula for the antipode in $\tT(\tT_+(V))$ expressed terms of Schröder trees (\cref{thm:antipodetT}). With the antipode formula, we study its implications in non-commutative probability via the Ebrahimi-Fard and Patras' framework. In particular, we are able to recover cumulant-moments formulas in terms of Schröder trees, which have recently appeared in the works \cite{JVNT,AC22}, as well as a new combinatorial formula that writes free Wick polynomials in terms of Schröder trees (\cref{thm:WickSch}). 

\subsection*{Organization of the paper} Besides the present section that serves as an introduction, the paper is organized in the following way. In \cref{sec:preliminaries}, we present a description of the main combinatorial objects used in this manuscript: trees, Schröder trees, and non-crossing partitions. We also explain a natural map associating a canonical non-crossing partition to any Schröder tree. In \cref{sec:doubletensor}, we precisely describe the Hopf algebra structure on the double tensor module $\tT(\tT_+(V))$ defined by Ebrahimi-Fard and Patras. We also explain how the coproduct in this Hopf algebra can be split, producing two non-coassociative coproducts, which will allow us to define three exponential-type maps on the dual of the double tensor Hopf algebra. \cref{sec:ncp} contains the definition of cumulants in non-commutative probability and describes the link between the double tensor Hopf algebra and non-commutative probability. \cref{sec:antipodeSS} aims to describe the general method in \cite{Menous2018} to compute a forest-type formula for the antipode in a particular class of Hopf algebras. We then apply this method to obtain a combinatorial formula for the antipode in $\sS(\tT_+(V))$ in terms of Schröder trees. In \cref{sec:antipodeTT}, we present in \cref{thm:antipodetT} the main result of our manuscript: a cancellation-free formula for the antipode in the double tensor algebra in terms of Schröder trees. Finally, in \cref{sec:applications}, we apply \cref{thm:antipodetT} in the Ebrahimi-Fard and Patras' framework for non-commutative probability and recover recently known cumulant-moment formulas in terms of Schröder trees, as well as a new expression for Anshelevich's free Wick polynomials also in terms of Schröder trees.

\section{Preliminaries on Schröder trees and non-crossing partitions}
\label{sec:preliminaries}

The aim of this section is to describe the combinatorial objects that are fundamental for the present manuscript: Schröder trees, non-crossing partitions, and their relations. To this end, we set $\mathbb{N}=\{0,1,2, \ldots \}$ as the set of non-negative integers. For every $n \in \mathbb{N}$, we denote by $[n]$ the set $[n]:=\{1,2, \hdots, n\}$, with $[0]:=\emptyset$. Given a finite set $I$, its cardinality is denoted by $|I|$. For every integer $n$, then $\Abs(n) \in \mathbb{N}$ denotes its absolute value.

\subsection{Trees}

A \emph{tree} is a connected graph that has no cycles. A \emph{rooted tree} is a tree with a distinguished vertex called the \emph{root}. All trees in this work are rooted, so we will not distinguish between trees and rooted trees. The set of edges of a rooted tree possesses a natural orientation, following the opposite direction to the root. Given a vertex $v$ of a tree $t$, the vertex connected to $v$ in the direction to the root is called the \emph{parent} of $v$, and any vertex connected to $v$ by an edge oriented towards the root are called \emph{children} of $v$. We say that a tree $t$ is \emph{planar} if, for every vertex $v$ of $t$, the set of children of $t$ is endowed with a total order. \emph{Non-planar trees} refer to trees that are not planar.

\

A \emph{leaf} of a tree is a vertex with no children, and an \textit{internal vertex} of a tree is a vertex that is not a leaf. We denote $\operatorname{Vert}(t)$ and $\operatorname{Int}(t)$ the set of vertices and internal vertices of $t$, respectively. Also, for $v\in \operatorname{Vert}(t)$, we denote $\operatorname{succ}(v)$ the set of children of $v$. A \emph{non-planar forest} is a set of non-planar rooted trees. Similarly, an \emph{ordered planar forest} is an ordered sequence of planar trees.

\

Any (planar or non-planar) rooted tree $t$  can be regarded as a poset in the following way. The elements of the poset are given by $\operatorname{Vert}(t)$. Moreover, for $v,w\in\operatorname{Vert}(t)$, we define the partial order $v\leq w$ if and only if the unique path from the root of $t$ towards $w$ passes through $v$.  In the case that $v\leq w$, we say that $w$ is a \textit{descendant} of $v$. Notice that in this poset structure, the root of $t$ is the unique minimal element, and the leaves of $t$ are the maximal elements.

\

Given a forest $f$, the \emph{branching} of $f$ is the new tree $B_+(f)$ obtained by joining all roots of every tree in $f$ to a new vertex $r$, so the root of $B_+(f)$ is $r$. Every tree $t$ can be obtained as the branching $B_+(f)$ of a (possibly empty) forest $f$. For planar trees, we will also denote $B_+(f)$ to be the planar tree obtained by grafting from left to right the elements of the sequence $f=(t_1,\ldots,t_s)$ to a new common root $r$.

\

The \emph{empty tree} $\mathsf{I}$ is the unique tree with no internal vertex and just one leaf. Given $m \geq 1$, the \emph{$m$-th corolla} $C_{(m)}$ is the unique tree with one vertex and $m$ leaves:
\[C_{(m)}:= B_+(\underbrace{\mathsf{I}, \mathsf{I}, \hdots, \mathsf{I}}_{m \text{ times }}) = \underbrace{\begin{tikzpicture}[baseline={([yshift=-1.5ex]current bounding box.center)},scale=0.6,
level 1/.style={level distance=7mm,sibling distance=7mm},solid node/.style={circle,draw,inner sep=1,fill=black, minimum size=1pt}]
\node(0)[solid node,label=above:{}]{} 
child{node(1)[solid node]{}
}
child{node(2)[solid node]{}
}
child{[dotted] node(11){}
}
child{[black] node(111)[solid node]{}
};
\end{tikzpicture}}_{m \text{ leaves }}.\]

\
For any rooted tree $t$, the \emph{size} $|t|$ of $t$ is its number of vertices, i.e.~$|t|:=|\operatorname{Vert}(t)|$. In the same way, we denote by $\internal(t):=|\operatorname{Int}(t)|$ the number of internal vertices of $t$. Therefore, $|t|-\internal(t)$ is the number of leaves of the tree $t$. If $t=B_+(t_1, t_2, \ldots, t_k)$, then
\begin{equation}
\label{eq:sizet}
|t|= 1 + |t_1|  + \cdots +|t_k|\quad \mbox{and}\quad \internal(t) =1 + \internal(t_1) +\cdots + \internal(t_k).
\end{equation}
The \emph{tree factorial} $t!$ is recursively defined by setting 
\[
t!:=\begin{cases}
    1, & \text{ if } |t| = 1;\\
    |t|\,t_1! t_2!\cdots t_k!, & \text{ if } t=B_+(t_1, t_2, \hdots, t_k).
\end{cases}\]
Finally, if $f$ is a forest formed by the rooted trees $t_1,\ldots,t_m$, we set $f! := t_1!\cdots t_m!$. 

\subsection{Schröder trees}

A \textit{Schröder tree} is a planar rooted tree for which each internal vertex has at least two children. For every $n,k \in \mathbb{N}$ such that $1 \leq k \leq n$, the set of Schröder trees with $k$ internal vertex and $n+1$ leaves is denoted by $\mathsf{Sch}_k(n)$. Also, we denote
$\mathsf{Sch}(n) = \bigcup_{k\geq 1}\mathsf{Sch}_k(n)$ the set of Schröder trees with $n+1$ leaves, with $\mathsf{Sch}(0)$ to be the set only containing the single-vertex tree $\circ$, whose unique vertex is considered as  a leaf. The \textit{degree} of a Schröder tree $t$ is the integer $n$ such that $t\in\ST(n)$, and it is denoted by $\deg(t)$. In this work, every Schröder tree is represented with the root at the top, and the leaves at the bottom. Internal vertices will be depicted as black nodes, while leaves will be depicted as white.


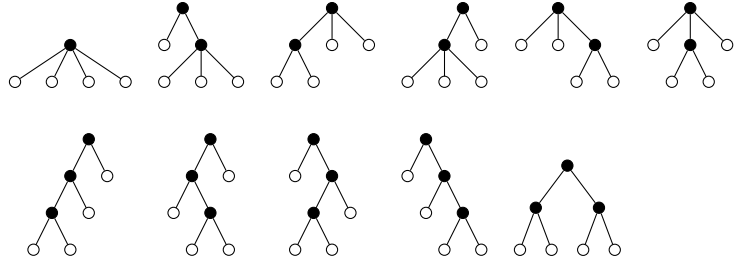
\begin{figure}[H]
\centering
\begin{tabular}{c c c c c c }
\begin{tikzpicture}[scale=0.7,
level 1/.style={level distance=7mm,sibling distance=7mm}]
\node(0)[solid node,label=above:{}]{} 
child{node(1)[hollow node]{}
}
child{node(2)[hollow node]{}
}
child{[black] node(11)[hollow node]{}
}
child{[black] node(111)[hollow node]{}
};
\end{tikzpicture} &  \begin{tikzpicture}[scale=0.7,
level 1/.style={level distance=7mm,sibling distance=7mm}]
\node(0)[solid node,label=above:{}]{} 
child{node(1)[hollow node]{}
}
child{node(2)[solid node]{}
	child{[black] node(11)[hollow node]{}}
	child{[black] node(111)[hollow node]{}}
	child{[black] node(1111)[hollow node]{}}
};
\end{tikzpicture} &
\begin{tikzpicture}[scale=0.7,
level 1/.style={level distance=7mm,sibling distance=7mm}]
\node(0)[solid node,label=above:{}]{} 
child{node(1)[solid node]{}
	child{[black] node(11)[hollow node]{}}
	child{[black] node(111)[hollow node]{}}
}
child{node(2)[hollow node]{}
}
child{node(3)[hollow node]{}};
\end{tikzpicture}  & \begin{tikzpicture}[scale=0.7,
level 1/.style={level distance=7mm,sibling distance=7mm}]
\node(0)[solid node,label=above:{}]{} 
child{node(1)[solid node]{}
	child{[black] node(11)[hollow node]{}}
	child{[black] node(111)[hollow node]{}}
	child{[black] node(1111)[hollow node]{}}
}
child{node(2)[hollow node]{}
};
\end{tikzpicture} & \begin{tikzpicture}[scale=0.7,
level 1/.style={level distance=7mm,sibling distance=7mm}]
\node(0)[solid node,label=above:{}]{} 
child{node(1)[hollow node]{}
}
child{node(2)[hollow node]{}
}
child{node(3)[solid node]{}
	child{[black] node(11)[hollow node]{}}
	child{[black] node(111)[hollow node]{}}};
\end{tikzpicture} &\begin{tikzpicture}[scale=0.7,
level 1/.style={level distance=7mm,sibling distance=7mm}]
\node(0)[solid node,label=above:{}]{} 
child{node(1)[hollow node]{}
	}
child{node(2)[solid node]{}
	child{[black] node(11)[hollow node]{}}
	child{[black] node(111)[hollow node]{}}
}
child{node(3)[hollow node]{}
};
\end{tikzpicture}  \\[0.2cm]
\begin{tikzpicture}[scale=0.7,
level 1/.style={level distance=7mm,sibling distance=7mm}]
\node(0)[solid node,label=above:{}]{} 
child{node(1)[solid node]{}
 	child{[black] node(11)[solid node]{}
		child{[black] node(112)[hollow node]{}}
		child{[black] node(1112)[hollow node]{}}
	}
	child{[black] node(111)[hollow node]{}}
	}
child{node(2)[hollow node]{}
};
\end{tikzpicture}  & \begin{tikzpicture}[scale=0.7,
level 1/.style={level distance=7mm,sibling distance=7mm}]
\node(0)[solid node,label=above:{}]{} 
child{node(1)[solid node]{}
 	child{[black] node(11)[hollow node]{}
	}
	child{[black] node(111)[solid node]{}
		child{[black] node(112)[hollow node]{}}
		child{[black] node(1112)[hollow node]{}}
	}	
	}
child{node(2)[hollow node]{}
};
\end{tikzpicture} & \begin{tikzpicture}[scale=0.7,
level 1/.style={level distance=7mm,sibling distance=7mm}]
\node(0)[solid node,label=above:{}]{} 
child{node(1)[hollow node]{}	
	}
child{node(2)[solid node]{}
	child{[black] node(11)[solid node]{}
		child{[black] node(112)[hollow node]{}}
		child{[black] node(1112)[hollow node]{}}
	}
	child{[black] node(111)[hollow node]{}	}
};
\end{tikzpicture}  & \begin{tikzpicture}[scale=0.7,
level 1/.style={level distance=7mm,sibling distance=7mm}]
\node(0)[solid node,label=above:{}]{} 
child{node(1)[hollow node]{}
	}
child{node(2)[solid node]{}
	child{[black] node(11)[hollow node]{}
	}
	child{[black] node(111)[solid node]{}
		child{[black] node(112)[hollow node]{}}
		child{[black] node(1112)[hollow node]{}}
	}
};
\end{tikzpicture} & \begin{tikzpicture}[scale=0.7,
level 1/.style={level distance=8mm,sibling distance=12mm},
level 2/.style={level distance=8mm,sibling distance=6mm},
]
\node(0)[solid node,label=above:{}]{} 
child{node(1)[solid node]{}
	child{[black] node(112)[hollow node]{}}
		child{[black] node(1112)[hollow node]{}}
	}
child{node(2)[solid node]{}
	child{[black] node(11)[hollow node]{}
	}
	child{[black] node(111)[hollow node]{}
	}
};
\end{tikzpicture}  
\end{tabular}
\vspace{0.5cm}
\caption{Schröder trees in $\mathsf{Sch}(3)$.}
\label{fig:st3}
\end{figure}

Analogously, a \textit{Schröder forest} is an ordered sequence of Schröder trees $F=(t_1,\ldots,t_m)$. In addition, the \textit{degree of $F$} is given by the sum of the degrees of its underlying trees, i.e.~$$\operatorname{deg}(F) = \deg(t_1)+\cdots +\deg(t_m).$$The set of Schröder forest of degree $n$ is then denoted by $\mathsf{FSch}(n)$, and the set of Schröder forest of the form $F= (t_1,\ldots,t_m)$, where $\deg(t_i)=n_i$ for $1\leq i\leq m$, is denoted by $\mathsf{FSch}(n_1,\ldots,n_m)$.

\ 

Schröder trees are known in the literature under several names, such as \emph{reduced planar trees} (see \cite{EFM2014}). For every $n\geq 0$, the number of elements in $\mathsf{Sch}_k(n)$ is given by
\begin{equation}\label{numberSch}
    |\mathsf{Sch}_k(n)|=\frac{1}{n+k+3}\binom{n+k+3}{k+1}\binom{n}{k} \quad , \quad 0 \leq k \leq n.
\end{equation}

\

This number corresponds to \cite[A033282]{oeis}. When $k=n$, a tree in $\mathsf{Sch}_{n}(n)$ corresponds to a \emph{binary tree}, a planar tree for which every internal vertex has exactly two children. The above formula leads to  $|\mathsf{Sch}_{n}(n)|=\mathsf{Cat}_n:=\frac{1}{n+1}\binom{2n}{n}$, the \emph{$n$-th Catalan number}. In general, Schröder trees are equinumerous to several objects in many contexts in mathematics. For instance, the set $\mathsf{Sch}_k(n)$ is in bijection with $k$-dimensional faces of the $n$-th dimensional associahedron, with $k$ non-crossing diagonals in a convex $(n+2)$-gon, or standard Young tableaux of shape $(k+1, k+1, 1^{n-k-1})$, just to name a few (see \cite{Simion2000} for an extensive list).

\

The \textit{skeleton} of a Schröder tree $t$ is the planar rooted subtree of $t$ generated by the internal vertices of $t$. We denote the skeleton of $t$ by $\sk(t)$. The skeleton of $t$ can also be regarded as a poset whose elements are given by $\operatorname{Int}(t)$. In the same way, if $F = (t_1,\ldots,t_m)$ is a Schröder forest, we define the \textit{skeleton of $F$}, denoted by $\sk(F)$, as the ordered forest $(\sk(t_1),\ldots,\sk(t_m))$. Moreover, $\sk(F)$ can also be regarded as the poset formed by the union of posets $\cup_{i=1}^m \sk(t_i).$ 

\ 

For our purposes, it is convenient to introduce two special subsets of $\mathsf{Sch}(n)$. 

\begin{definition}
    Let $t$ be a Schröder tree.
    \begin{enumerate}
        \item We say that $t$ is a \textit{prime Schröder tree} if its leftmost subtree is a leaf. In other words, if $t = B_+(t_1,\ldots,t_k)$, then $t_1 = \circ$ (a single-vertex tree). We denote $\PST(n)$ the set of prime Schröder trees with $n+1$ leaves.
        \item We say that $t$ is a \textit{Boolean Schröder tree} if $t$ is a Schröder tree such that for every $v\in\operatorname{Int}(t)$, the children of $v$ are all leaves except possibly its leftmost child. For each $n\geq1$, we denote $\mathsf{BSch}(n)$ the set of Boolean Schröder trees with $n+1$ leaves.
    \end{enumerate}
\end{definition}

It is known that prime Schröder trees are counted by the \emph{large Schröder numbers} \cite[A006318]{oeis}. In particular, for any $n \geq 1$, there exists a two-to-one surjective function \linebreak $P:\PST(n)\to\mathsf{Sch}(n-1)$, so that $|\PST(n)| = 2|\mathsf{Sch}(n-1)|$. For the case of Boolean Schröder trees, it can be shown that $|\mathsf{BSch}(n)| = 2^{n-1}$, for any $n\geq1$.

\begin{figure}[H]
\centering
\begin{tabular}{c c c c  }
\begin{tikzpicture}[scale=0.7,
level 1/.style={level distance=7mm,sibling distance=7mm}]
\node(0)[solid node,label=above:{}]{} 
child{node(1)[hollow node]{}
}
child{node(2)[hollow node]{}
}
child{[black] node(11)[hollow node]{}
}
child{[black] node(111)[hollow node]{}
};
\end{tikzpicture}  &
\begin{tikzpicture}[scale=0.7,
level 1/.style={level distance=7mm,sibling distance=7mm}]
\node(0)[solid node,label=above:{}]{} 
child{node(1)[solid node]{}
	child{[black] node(11)[hollow node]{}}
	child{[black] node(111)[hollow node]{}}
}
child{node(2)[hollow node]{}
}
child{node(3)[hollow node]{}};
\end{tikzpicture}  & \begin{tikzpicture}[scale=0.7,
level 1/.style={level distance=7mm,sibling distance=7mm}]
\node(0)[solid node,label=above:{}]{} 
child{node(1)[solid node]{}
	child{[black] node(11)[hollow node]{}}
	child{[black] node(111)[hollow node]{}}
	child{[black] node(1111)[hollow node]{}}
}
child{node(2)[hollow node]{}
};
\end{tikzpicture} &  $\quad$\begin{tikzpicture}[scale=0.7,
level 1/.style={level distance=7mm,sibling distance=7mm}]
\node(0)[solid node,label=above:{}]{} 
child{node(1)[solid node]{}
 	child{[black] node(11)[solid node]{}
		child{[black] node(112)[hollow node]{}}
		child{[black] node(1112)[hollow node]{}}
	}
	child{[black] node(111)[hollow node]{}}
	}
child{node(2)[hollow node]{}
};
\end{tikzpicture}
  
\end{tabular}
\vspace{0.5cm}
\caption{Boolean Schröder trees in $\mathsf{BSch}(3)$.}
\label{fig:bst3}
\end{figure}
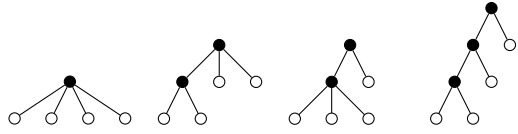

Let $t \in \mathsf{Sch}(n)$ be a Schröder tree with $n+1$ leaves. Since $t$ is planar, we can naturally label the set of leaves of $t$ with the numbers $1,2,\ldots, n+1$, reading the leaves from left to right. Hence, two leaves of $t$ are \emph{consecutive} if they are labelled with $i$ and $i+1$, for some $1 \leq i \geq n$. A \emph{sector}\footnote{Our definition comes from \cite{PatrasSchocker}, where sectors are called ``branchings''. } of $t$ is a minimal subtree of $t$ containing two consecutive leaves. Roughly speaking, a sector of $t$ is the area encompassed by two consecutive leaves of $t$. Also, notice that the root of a sector is the intersection of the unique paths from the consecutive leaves to the root of $t$. In addition, if $v$ is an internal vertex of $t$ and $s$ is a sector of $t$ its root is $v$, we say that $s$ is \emph{adjacent to the vertex $v$}. 

\

We denote by $\operatorname{Sect}(t)$ the set of all sectors of $t$. It is clear that $|\operatorname{Sect}(t)| = n$. The set $\operatorname{Sect}(t)$ is naturally endowed with a total order, consisting of listing each sector of $t$ from left to right, bottom to top. Formally:
\begin{itemize}
    \item[$\bullet$] If $t=\circ$, we have an empty set of sectors. If $t=C_{(m+1)}$ is a corolla with exactly $m+1$ leaves, for some $m \geq 0$, then $t$ possess $m$ sectors. We denote by $s_1, s_2, \hdots, s_m$ the sectors of $t$, by order of appearance from left to right. Then, we define the following total order between the sectors:
    \[s_1 < s_2 < \cdots < s_m.\]
    \item[$\bullet$] More generally, assume that  $t=B_+(t_0, t_1, \hdots, t_m)$, where $t_0, t_1, \hdots, t_m$ are Schröder trees. 
    If $s$ is a sector of $t$ adjacent to the root of $t$ and located between the trees $t_{{i}}$ and $t_{{i+1}}$, we extend inductively the orders of $\operatorname{Sect}(t_0), \operatorname{Sect}(t_1), \ldots, \operatorname{Sect}(t_m)$ to an order on $\operatorname{Sect}(t)$, via
    \[s' < s < s'', \quad \text{ for all }  s' \in \operatorname{Sect}(t_i), \, s'' \in \operatorname{Sect}(t_{i+1}). \]
\end{itemize}

Let $t \in \mathsf{Sch}(n)$, with $n \geq 1$. There exists a unique increasing bijection 
\begin{equation}
\iota_t: ([n], \leq_{\mathbb{N}}) \to (\operatorname{Sect}(t), <),
\end{equation}
where $\leq_{\mathbb{N}}$ is the usual order on $\mathbb{N}$. We call $\iota_t$ the \emph{natural labelling} of $t$.

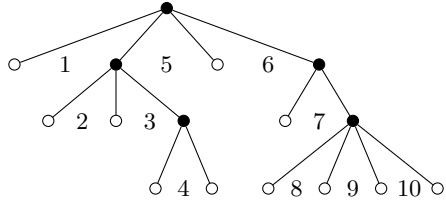
\begin{figure}[H]
\centering
\begin{subfigure}{.4\textwidth}
  \centering
  \begin{tikzpicture}[scale=1.5,font=\footnotesize,
 level 1/.style={level distance=5mm,sibling distance=9mm},
 level 2/.style={level distance=5mm,sibling distance=6mm},
 level 3/.style={level distance=6mm,sibling distance=5mm},
 level 4/.style={level distance=10mm,sibling distance=5mm},
]
\node(0)[solid node,label=above:{}]{} 
child{node(1)[hollow node]{}
}
child{node(2)[solid node]{}
child{[black] node(11)[hollow node]{}}
child{[black] node(12)[hollow node]{}}
child{[black] node(13)[solid node]{}
	child{[black] node(131)[hollow node]{}}
	child{[black] node(132)[hollow node]{}}
}
edge from parent node[left]{}
}
child{node(3)[hollow node]{}
}
child{node(4)[solid node]{}
child{[black] node(41)[hollow node, ]{}}
child{[black] node(42)[solid node,]{}
	child{[black] node(421)[hollow node, ]{}}
	child{[black] node(422)[hollow node,]{}}
	child{[black] node(423)[hollow node, ]{}}
	child{[black] node(424)[hollow node,]{}}
}
edge from parent node[right]{}
};

\path (1) -- node {$1$} (2);
\path (2) -- node {$5$} (3);
\path (3) -- node {$6$} (4);

\path (11) -- node (H) {${2}$} (12);
\path (12) -- node {$3$} (13);
\path (131) -- node {$4$} (132);
\path (41) -- node {$7$} (42);
\path (421) -- node {$8$} (422);
\path (422) -- node {$9$} (423);
\path (423) -- node {$10$} (424);

\end{tikzpicture}
\end{subfigure}
\caption{A Schröder tree $t\in \ST(11)$ and its natural labelling.}
\label{fig:TotalOrderingSch} 
\end{figure}

\subsection{Set partitions}

Let $I$ be a finite set. A \emph{partition} of $I$ (or \emph{set partition}) is a finite set formed by non-empty sets $\pi=\{B_1, B_2, \hdots, B_k\}$, such that the union of all elements in $\pi$ is $I$, and every pair of distinct elements in $\pi$ are disjoint. Every element $B \in \pi$ is called a \emph{block} of $\pi$. If $\Pi(I)$ denotes the set of all partitions of $I$ and $\pi \in \Pi(I)$, we write $\pi \vdash I$. There is a unique partition of $\Pi(\emptyset)$, with no blocks.

\

Let $\pi = \{B_1, B_2, \hdots, B_k\}$ be a partition of $I$. Given a subset $S$ of $I$, the \emph{restriction} of $\pi$ to $S$ is the partition given by $\pi_S:= \{B_1 \cap S, B_2 \cap S, \hdots, B_k \cap S\}^{\text{\faEraser}}$, where the symbol \faEraser \, indicates that empty intersections have been omitted. The \emph{length} of $\pi$, $|\pi|:=k$, is defined as its number of blocks. The \emph{factorial} of the partition $\pi$ is given by
\[\pi!:=\prod_{B \in \pi} (\# B)!.\]

For every finite set $I$, the set $\Pi(I)$ is partially ordered under \emph{reverse refinement}: we set $\pi \leq \sigma$ if every block of $\pi$ is contained in a block of $\sigma$. In other words, $\pi \leq \sigma$ if every block of $\sigma$ can be expressed as union of blocks of $\pi$. The partition $0_I$ of $I$ into singletons is the unique minimum, and $1_I:=\{I\}$ is the unique maximum under the refinement order. Therefore, $(\Pi_I, \leq)$ is a lattice.

\

Let $\pi, \sigma \vdash I$ such that $\pi \leq \sigma$. If $B \in \sigma$, let $n_B:=|\pi_B|$. That is, $n_B$ is the number of blocks of $\pi$ that refine the block $B$ of $\sigma$. The Möbius function of $\Pi(I)$ is then given by (\cite[Chap. 3]{stanley2011enumerative})
\begin{equation}
    \operatorname{M\ddot{o}b}(\pi, \sigma) = (-1)^{|\pi|-|\sigma|}\prod_{B  \, \in \,  \sigma}(n_B-1)!,
\end{equation}
for every $\pi, \sigma \vdash I$ such that $\pi \leq \sigma$.

\subsection{Non-crossing partitions} Let $I$ be a finite totally ordered set.  A \emph{non-crossing partition} of $I$ is a partition $\pi \in \Pi(I)$ for which whenever four elements $a < b < c <d$ in $I$ are such that if $a,c$ are in a same block of $\pi$ and $b, d$ are in another block of $\pi$, then both blocks coincide. The set of non-crossing partitions of $I$ is denoted by $\mathsf{NC}(I)$. When $I= [n]$, we simply denote $\mathsf{NC}(n)$ the set of non-crossing partitions of $[n]$.

\

The reverse refinement order $\leq$ on $\Pi(n)$ induces a poset structure on $\mathsf{NC}(n)$. Since the finest partition $0_n$ and the coarsest partition $1_n$ in $\Pi(n)$ are non-crossing partitions, then $(\mathsf{NC}(n), \leq)$ is a lattice. This poset is closed under meet operation, but not under join. Hence, it is not a sublattice of $(\Pi(n), \leq)$, but rather a meet-sublattice.

\

Non-crossing partitions belong to a vast family of combinatorial objects, namely the \emph{Catalan family}. The number of non-crossing partitions in $\mathsf{NC}(n)$ is the $n$-th Catalan number $\mathsf{Cat_n}$. 
 Several objects are related to non-crossing partitions in many contexts (see \cite{Simion2000} for a compendium of some related objects, and \cite{Stanley2015Catalan} for a more extensive list).

\

One of the first studies on non-crossing partitions goes back to 1952 in the work of Becker (see \cite{Becker1952}), where non-crossing partitions are called ``\emph{planar rhyme schemes}''. In 1972, Kreweras \cite{Kreweras1972} and then Poupard \cite{Poupard1972} continued with a combinatorial analysis of these objects. For instance, Kreweras obtained a formula for the Möbius function of $\mathsf{NC}(n)$ between its minimal and maximal elements: 
\begin{equation}
    \operatorname{M\ddot{o}b}_{\mathsf{NC}(n)}(0_n, 1_n) = (-1)^{n-1}\mathsf{Cat_{n-1}}.
\end{equation}

Several types of non-crossing partitions are considered in this work, arising from non-commutative probability theory. An \emph{interval partition of} $[n]$ is a non-crossing partition $\pi \in \mathsf{NC}(n)$ such that all its blocks are of the form $\{k, k + 1,\ldots, k + l\}$ for some $1 \leq k \leq n$ and $0 \leq l \leq n - k$.  
We denote by $\mathsf{NCInt}(n)$ the set of interval partitions of $[n]$. It is straightforward to verify that the reverse refinement order also provides $\mathsf{NCInt}(n)$ with a poset structure, which is isomorphic to the poset of subsets of a set of cardinality $n-1$.

\allowdisplaybreaks

\begin{figure}[H]
$$
\begin{array}{c c c}
   \begin{tikzpicture}[thick,font=\small]
     \path 	(0,0) 		node (a) {1}
           	(0.5,0) 	node (b) {2}
           	(1,0) 		node (c) {3}
           	(1.5,0) 	node (d) {4}
           	(2,0) 		node (e) {5}
           	(2.5,0) 	node (f)  {6};
     \draw (a) -- +(0,0.75) -| (d);
     \draw (c) -- +(0,0.75) -| (c);
     \draw (b) -- +(0,0.60) -| (b);
     \draw (e) -- +(0,0.75) -| (f);
   \end{tikzpicture}
    & \quad &
     \begin{tikzpicture}[thick,font=\small]
     \path 	(0,0) 		node (a) {1}
           	(0.5,0) 	node (b) {2}
           	(1,0) 		node (c) {3}
           	(1.5,0) 	node (d) {4}
           	(2,0) 		node (e) {5}
           	(2.5,0) 	node (f)  {6};
     \draw (a) -- +(0,0.75) -| (b);
     \draw (c) -- +(0,0.75) -| (e);
     \draw (d) -- +(0,0.75) -| (d);
      \draw (f) -- +(0,0.75) -| (f);  
       \end{tikzpicture} \\
   \mbox{(a) Non-crossing partition}& \quad & \mbox{(b) Interval partition}
\end{array}
$$
\caption{Different types of set partitions of the set $[6]=\{1,\ldots,6\}$. The blocks of the partitions are represented by arcs. The non-crossing condition of the blocks means that there are no intersections of the arcs.}
\end{figure}
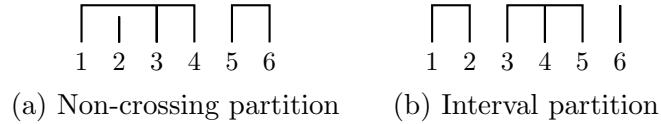

\tikzset{
solid node/.style={circle,draw,inner sep=1.5,fill=black, minimum size=1pt},
hollow node/.style={circle,draw,inner sep=1.5,fill=white,  minimum size=1pt}
}

\subsection{Nesting forests from non-crossing partitions}

To every non-crossing partition $\pi \in \mathsf{NC}(n)$ with $k$ blocks, we associate to it a (non-planar) forest $f_\pi$ with $k$ vertices in total, constructed from a hierarchy law of the blocks of $\pi$. More precisely, given a non-crossing partition $\pi$ and two blocks $B, B' \in \pi$, we say that $B$ is \emph{nested} on $B'$ if $\min B' \leq x \leq \max B'$, for every $x \in B$. This relation gives a poset structure on the set of blocks of $\pi$. The minimal elements are called \emph{outer blocks} of $\pi$. By construction, the poset has no cycles and may have several connected components. Hence, the nesting poset associated to $\pi$ is a (non-planar) forest, denoted by $f_{\pi}$, and called the \emph{forest of nestings} of $\pi$. In particular, $f_{\pi}$ is a single tree if and only if $\pi$ is an \emph{irreducible non-crossing partition}, i.e.~$1$ and $n$ belong to the same block in $\pi$.

\ 

If $\{B_1,\ldots,B_s\}$ is the set of outer blocks of a non-crossing partition $\pi$, we can construct $\pi_1,\ldots,\pi_s$ irreducible non-crossing partitions as
\[\pi_i := \{B\in \pi\,:\,B\mbox{ is nested in }B_i\},\] 
for each $1\leq i\leq s$. The non-crossing partitions $\pi_1,\ldots,\pi_s$ are called the \textit{irreducible components of $\pi$}. Hence, the nesting forest associated to $\pi$ is given by $f_{\pi}=\{f_{\pi_1}, \hdots, f_{\pi_s}\}$.

\

In the following picture we draw two non-crossing partitions $\pi$ and $\sigma$ together with their respective forest of nesting $f_{\pi}$ and $f_{\sigma}$. For clarity, each vertex of the forest is decorated with the minimal element of their associated block in the partition.

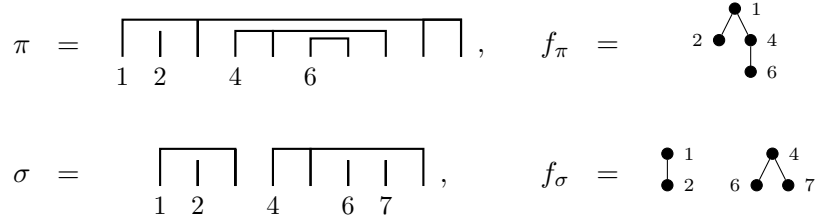
\begin{center}
\begin{figure}[H]
$$
\begin{array}{c c c c c c}
\pi &=& \begin{tikzpicture}[baseline={([yshift=0.3ex]current bounding box.center)},thick,font=\small]
     \path 	(0,0) 		node (a) {1}
           	(0.5,0) 	node (b) {2}
           	(1,0) 		node (c) {\phantom{1}}
           	(1.5,0) 	node (d) {4}
           	(2,0) 		node (e) {\phantom{1}}
           	(2.5,0) 	node (f)  {6}
           	(3,0) 		node (g) {\phantom{1}}
           	(3.5,0) 	node (h) {\phantom{1}}
		(4,0) 		node (i) {\phantom{1}}
		(4.5,0) 	node (j) {\phantom{1}};
     \draw (a) -- +(0,0.75) -| (j);
     \draw (b) -- +(0,0.60) -| (b);
     \draw (c) -- +(0,0.75) -| (c);
     \draw (d) -- +(0,0.60) -| (e);
     \draw (e) -- +(0,0.60) -| (h);
     \draw (f) -- +(0,0.50) -| (g);
     \draw (i) -- +(0,0.75) -| (j);
   \end{tikzpicture}, 	& \quad f_{\pi} &=&
				  \begin{tikzpicture}[baseline={([yshift=-1ex]current bounding box.center)},scale=0.6,
level 1/.style={level distance=7mm,sibling distance=7mm}]
\node(0)[solid node,label=right:{\tiny 1}]{} 
child{node(1)[solid node,label=left:{\tiny 2}]{}
}
child{node(2)[solid node,label=right:{\tiny 4}]{}
	child{[black] node(11)[solid node,label=right:{\tiny 6}]{}}
};
\end{tikzpicture}
\\[1cm]
\sigma &= & \begin{tikzpicture}[baseline={([yshift=0.3ex]current bounding box.center)},thick,font=\small]
     \path 	(0,0) 		node (a) {1}
           	(0.5,0) 	node (b) {2}
           	(1,0) 		node (c) {\phantom{1}}
           	(1.5,0) 	node (d) {4}
           	(2,0) 		node (e) {\phantom{1}}
           	(2.5,0) 	node (f)  {6}
           	(3,0) 		node (g) {7}
           	(3.5,0) 	node (h) {\phantom{1}};
     \draw (a) -- +(0,0.75) -| (c);
     \draw (b) -- +(0,0.60) -| (b);
     \draw (c) -- +(0,0.75) -| (c);
     \draw (d) -- +(0,0.75) -| (e);
     \draw (e) -- +(0,0.75) -| (h);
     \draw (f)  -- +(0,0.60) -|(f);
     \draw (g) -- +(0,0.60)-|(g); 
   \end{tikzpicture}, 	& \quad f_{\sigma} &=&\;\;\begin{tikzpicture}[baseline={([yshift=-1ex]current bounding box.center)},scale=0.6,
level 1/.style={level distance=7mm,sibling distance=7mm}]
\node(0)[solid node,label=right:{\tiny 1}]{} 
child{node(1)[solid node,label=right:{\tiny 2}]{}
};
\end{tikzpicture}
\;\; \begin{tikzpicture}[baseline={([yshift=-1ex]current bounding box.center)},scale=0.6,
level 1/.style={level distance=7mm,sibling distance=7mm}]
\node(0)[solid node,label=right:{\tiny 4}]{} 
child{node(1)[solid node,label=left:{\tiny 6}]{}
}
child{node(2)[solid node,label=right:{\tiny 7}]{}
};
\end{tikzpicture}
\end{array}  
$$
\end{figure}
\end{center}
It is convenient to consider the case in which the partial order described by the forest of nestings of a non-crossing partition can be extended to a total order. More precisely, we say that $(\pi,\lambda)$ is a \emph{monotone partition on $[n]$} if $\pi\in\NC(n)$ and $\lambda:\pi\to[|\pi|]$ is a order-preserving bijection, i.e.~if $B$ and $B'$ are distinct blocks of $\pi$ such that $B'$ is nested in $B$, then $\lambda(B)< \lambda(B')$. The set of monotone partitions on $[n]$ is denoted by $\mathcal{M}(n)$.

\subsection{From Schröder trees to non-crossing partitions.}

Let $t \in \mathsf{Sch}(n)$. Recall that the set $\operatorname{Sect}(t)$ of all sectors of $t$ is naturally endowed with a total order. Moreover, the Schröder tree $t$ has a natural labelling $\iota_t: [n] \to \operatorname{Sect}(t)$, which is a bijection between totally ordered sets. The next definition presents a natural way to associate a non-crossing partition to any Schr\"oder tree. 

\begin{definition}
\label{def:pit}
Let $t\in\ST(n)$, with set of internal vertices $\operatorname{Int}(t) = \{v_1, v_2, \hdots, v_k\}$. The \emph{non-crossing partition associated to $t$} is the set partition $\pi(t):=\{B_1, \hdots, B_k\} \in \mathsf{NC}(n)$ formed by $k$ blocks, given by
\begin{equation}
    \label{eq:defipit}
    B_i:=\{\iota^{-1}_t(s): s \text{ is a sector adjacent to } v_i\},
\end{equation}
for each $1 \leq i \leq k$. In the context of \eqref{eq:defipit}, we also say that $B_i$ is \textit{associated to $v_i$}.
\end{definition}

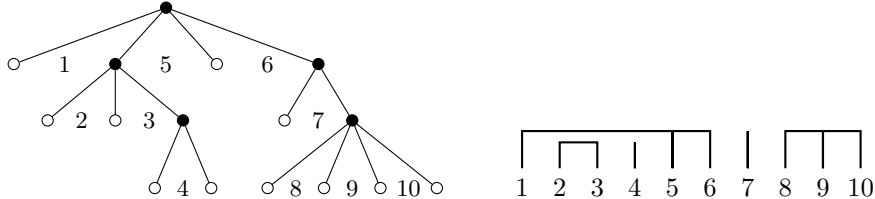
\begin{figure}[H]
\centering
\begin{subfigure}{.4\textwidth}
  \centering
  \begin{tikzpicture}[scale=1.5,font=\footnotesize,
 level 1/.style={level distance=5mm,sibling distance=9mm},
 level 2/.style={level distance=5mm,sibling distance=6mm},
 level 3/.style={level distance=6mm,sibling distance=5mm},
 level 4/.style={level distance=10mm,sibling distance=5mm},
]
\node(0)[solid node,label=above:{}]{} 
child{node(1)[hollow node]{}
}
child{node(2)[solid node]{}
child{[black] node(11)[hollow node]{}}
child{[black] node(12)[hollow node]{}}
child{[black] node(13)[solid node]{}
	child{[black] node(131)[hollow node]{}}
	child{[black] node(132)[hollow node]{}}
}
edge from parent node[left]{}
}
child{node(3)[hollow node]{}
}
child{node(4)[solid node]{}
child{[black] node(41)[hollow node, ]{}}
child{[black] node(42)[solid node,]{}
	child{[black] node(421)[hollow node, ]{}}
	child{[black] node(422)[hollow node,]{}}
	child{[black] node(423)[hollow node, ]{}}
	child{[black] node(424)[hollow node,]{}}
}
edge from parent node[right]{}
};

\path (1) -- node {$1$} (2);
\path (2) -- node {$5$} (3);
\path (3) -- node {$6$} (4);

\path (11) -- node (H) {${2}$} (12);
\path (12) -- node {$3$} (13);
\path (131) -- node {$4$} (132);
\path (41) -- node {$7$} (42);
\path (421) -- node {$8$} (422);
\path (422) -- node {$9$} (423);
\path (423) -- node {$10$} (424);

\end{tikzpicture}
\end{subfigure}%
\begin{subfigure}{.4\textwidth}
  \centering
  \begin{tikzpicture}[thick,font=\small]
     \path 	(0,0) 		node (a) {1}
           	(0.5,0) 	node (b) {2}
           	(1,0) 		node (c) {3}
           	(1.5,0) 	node (d) {4}
           	(2,0) 		node (e) {5}
           	(2.5,0) 	node (f)  {6}
           	(3,0) 		node (g) {7}
           	(3.5,0) 	node (h) {8}
           	(4,0) 		node (i) {9}
           	(4.5,0) 	node (j)  {10};
     \draw (a) -- +(0,0.75) -| (f);
     \draw (e) -- +(0,0.75) -| (e);
     \draw (b) -- +(0,0.60) -| (c);
     \draw (d) -- +(0,0.60) -| (d);
     \draw (g) -- +(0,0.75) -| (g);
     \draw (h) -- +(0,0.75) -| (j);
     \draw (i) -- +(0,0.75) -| (i);
   \end{tikzpicture}
\end{subfigure}
\caption{Schröder tree $t$ with $10$ sectors, and its associated non-crossing partition $\pi(t) = \big\{\{ 1,5,6\},\{2,3\},\{4\}, \{7\}, \{8,9,10\}\big\}\in \NC(10)$.}
\label{fig:DecSchTree} 
\end{figure}

It is straightforward to verify that the map $\mathsf{Sch}(n)\ni t\mapsto \pi(t) \in \mathsf{NC}(n)$ is a surjection. In fact, if we restrict this map to the subset of prime Schröder trees $\mathsf{PSch}(n)$, we still get a surjection that is related to the Möbius function on $\mathsf{NC}(n)$.

\begin{proposition}[{\cite[Eq. (111)]{JVNT}}]
    \label{prop:NumberPST}
    Let $n\geq1$. For any $\pi\in\NC(n)$, we have $$|\{t\in\mathsf{PSch}(n)\,:\, \pi(t) = \pi \}| = \Abs\big(\!\operatorname{M\ddot{o}b}_{\mathsf{NC}(n)}(\pi,1_n)\big).$$
\end{proposition}
To finish this section, we remark that the map $t \mapsto \pi(t)$ actually provides a bijection between $\mathsf{BSch}(n)$ and $\mathsf{NCInt}(n)$, for any $n\geq 1$. As a consequence, we can conclude that $| \mathsf{BSch}(n)| = |\mathsf{NCInt}(n)| = 2^{n-1}$.

\section{Hopf algebra structure on the double tensor module}
\label{sec:doubletensor}

In this section, we present the definition of the double tensor Hopf algebra introduced by Ebrahimi-Fard and Patras in \cite{EFP:2015}. We also describe how the coproduct can be split into two parts, generating a shuffle (or dendriform) structure on the dual of the double tensor Hopf algebra. Finally, we explain how the shuffle structure allows us to construct three different bijections which resemble the relation between a Lie group and its Lie algebra.

\

We denote by $\mathbb{K}$ a base field of characteristic $0$ and by $\mathbb{C}$ the field of complex numbers. The vector spaces considered in this work are $\mathbb{C}$-vector spaces. We mention that the constructions presented in this section can also be done considering a base field $\mathbb{K}$.

\subsection{Definition of the double tensor Hopf algebra} Let $V$ be a vector space. If $k \geq 0$, we write elementary tensors
from $V^{\otimes k}$ as words, $u_1u_2\cdots u_k$, with $u_i \in V$ for $1\leq i\leq k$. Also, we identify $V^{\otimes 0}$ with $\mathbb{C}$. Then, we call the vector spaces of linear combinations of words and non-empty words \[\mathsf{T}(V):=\bigoplus_{k \, \geq \, 0}V^{\otimes k} \qquad \text{ and } \qquad \mathsf{T}_+(V):=\bigoplus_{k \, \geq \, 1}V^{\otimes k}\]
as the \emph{tensor module} and the {\it augmented tensor module} generated by $V$, respectively.

\

It also makes sense to consider the tensor module of $\tT_+(V)$, denoted by $\tT(\tT_+(V))$. To avoid confusion, the second tensor product appearing in the construction of the tensor module of $\tT_+(V)$ is denoted by a bar, i.e.~if $w_1,w_2,\ldots,w_n\in \tT_+(V)$, then we write
$$w_1|w_2|\cdots| w_n \in \tT_+(V)^{\otimes n}.$$

As a vector space, $\tT(\tT_+(V))$ is graded, where for any $n\geq0$, the $n$-th graded component is given by

\begin{equation}
    \label{eq:grading}
    \mathsf{T}(\mathsf{T}_+(V))_n:= \bigoplus_{\substack{n_1 + \cdots + n_k = n\\n_i\geq0}}V^{\otimes n_1} \oplus \cdots \oplus V^{\otimes n_k}.
\end{equation}

It turns out that the tensor module $\tT(\tT_+(V))$ can be endowed with a product and a coproduct respecting the grading \eqref{eq:grading} as follows:

\


$\bullet$ {\textit{Product rule}}: if $v= v_1|\cdots |v_r \in \mathsf{T}_+(V)^{\otimes r}$ and $w= w_1|\cdots |w_s \in \mathsf{T}_+(V)^{\otimes s}$, then the product of $v$ and $w$ is the element
\[v|w:= v_1|\cdots |v_r|w_1|\cdots| w_s \in  \mathsf{T}_+(V)^{\otimes(r+s)}.\]
In addition, if $\uno$ stands for the empty word, then we identify $w|\uno = w = \uno|w$ for any $w\in \tT_+(V)$ so that $\uno$ works as a unit.

\

$\bullet$ {\textit{Coproduct rule}}: the definition of the coproduct involves extra technical notions. Given a word $w=u_1 \cdots u_n \in V^{\otimes n}$ and $I=\{i_1, \hdots, i_k\}\subset \mathbb{N}$ where $i_1<\cdots <i_k$, we write the \textit{restriction of $w$ to $I$} by 
\begin{equation}
    w_I:=u_{i_1}\cdots u_{i_k}.
\end{equation} 
If $I \cap [n]=\emptyset$, we set $w_I:=\uno$. 

\

A set $I \subseteq \mathbb{N}$ is an {\it interval} if $I$ is a set of consecutive positive integers. For a pair of nested intervals  $I \subseteq J \subseteq [n]$, the set $J\setminus I$ can be decomposed (uniquely) into the disjoint union 
\[J \setminus I = K_1 \sqcup K_2 \sqcup \cdots \sqcup K_r,\] where $K_1, \hdots, K_r$ are intervals of $[n]$, such that no union of two such blocks is an interval. Also, the above list is ordered increasingly according to the minimal elements. Therefore, $\min K_1 < \cdots < \min K_r$. We call the sequence $K(I,J):=\{K_1, \hdots, K_r\}$ the \emph{the decomposition of $J \setminus I$ into its connected components}. With this, we set
\begin{equation}
    w^{(I,J)}:=w_{K_1}| \cdots | w_{K_r}\in \tT(\tT_+(V)).
\end{equation}

If $J=[n]$, we set $K(I):= K(I,[n])$ and $w^{(I)}:=w^{(I,J)}$. 

\begin{example}
Let $V$ be a vector space. For $n=9$ and $w=u_1 \cdots u_9\in V^{\otimes 9}$, let $J=\{2,3,4,6,7,8,9\}$ and $I=\{4,7,9\}$, so that $J\setminus I=\{2,3,6,8\}$. Then, the elements $w^{(I,J)}$, $w^{(I)}$ and $w^{(J)}$ are, respectively,
\[w^{(I,J)}=u_2u_3 |u_6|u_8, \qquad  w^{(I)}=u_1u_2u_3 | u_5u_6 | u_8 \quad \text{ and } \quad w^{(J)} = u_1|u_5.\]
\end{example}
With this notation, define now the map $\Delta: \mathsf{T}_+(V) \to \mathsf{T}(V)\otimes\mathsf{T}(\mathsf{T}_+(V))$ by

\begin{eqnarray}
 \label{eq:defcoproduct}   \Delta(w):&=&\sum_{I \, \subseteq \, [n]}w_I \otimes w^{(I)}\\
&=&\sum_{I \, \subseteq \, [n]}w_I \otimes w_{K_1}|\cdots | w_{K_r}, \nonumber
\end{eqnarray}
where $K(I) = \{K_1, \hdots, K_r\}$.

\

To complete the definition, we extend the map $\Delta$ multiplicatively to all of $\mathsf{T}(\mathsf{T}_+(V))$, by setting $\Delta(\uno) = \uno\otimes \uno$ and
\[\Delta(w_1|\cdots |w_k):=\Delta(w_1)| \cdots | \Delta(w_k),\]
for any $w_1,\ldots,w_k\in \tT_+(V).$ Finally, we define the counit map $\epsilon:\tT(\tT_+(V))\to\mathbb{C}$ as the algebra morphism given by $\epsilon(\uno)=1$ and 
$\epsilon(w) = 0$ for any $w\in \tT_+(\tT_+(V))$.

\begin{example}
Let $V$ be a vector space and $a_1,a_2,a_3\in V$. Then, we have the following computations for the coproduct in \eqref{eq:defcoproduct}:
\begin{eqnarray*}
    \Delta(a_1) &=& a_1\otimes \uno + \uno\otimes a_1,
    \\ \Delta(a_1a_2) &=& a_1a_2\otimes \uno + a_1\otimes a_2 + a_2\otimes a_1 + \uno\otimes a_1a_2,
    \\ \Delta(a_1a_2a_3) &=& a_1a_2a_3\otimes \uno + a_1a_2\otimes a_3 + a_1a_3\otimes a_2 + a_2a_3\otimes a_1 \\ & & + a_1\otimes a_2a_3 + a_2\otimes a_1|a_3 + a_3\otimes a_1a_2 + \uno\otimes a_1a_2a_3.
\end{eqnarray*}
\end{example}

The iterated coproduct is recursively defined by $\Delta^{[m+1]} := (\id^{\otimes (m-1)}\otimes \Delta)\circ \Delta^{[m]}$ for $m\geq1$, where $\id$ stands for the identity map on $\tT(\tT_+(V))$, $\Delta^{[1]} = \id$ and $\Delta^{[2]} = \Delta$. Then, for $m\geq2$ and a non-empty word $w = u_1\cdots u_n \in \tT_+(V)$, we have:

\begin{equation}\label{iterated}
\Delta^{[m]}(w)= \sum_{I=I_1 \, \subseteq \, I_2 \, \subseteq \cdots \, \subseteq I_{m-1} \, \subseteq [n]}w_I \otimes w^{(I_1, I_2)}\otimes w^{(I_2, I_3)}\otimes \cdots \otimes w^{(I_{m-1})}.
\end{equation} 

It turns out that the tensor module $\tT(\tT_+(V))$ endowed with the bar (concatenation) product and the coproduct described in \eqref{eq:defcoproduct} is a graded bialgebra which in addition is \textit{connected}, i.e.~$\tT(\tT_+(V))_0 \cong \mathbb{K}$. Since any connected bialgebra is a Hopf algebra (see, for instance, \cite[Lem. 7.6.2]{Radford}), we have the following result.

\begin{theorem}[\cite{EFP:2015}]
    The tensor module $\tT(\tT_+(V))$ endowed with the bar product and the coproduct in \eqref{eq:defcoproduct} is a graded connected non-commutative and non-cocommutative Hopf algebra.
\end{theorem}

\begin{remark}
    The Hopf algebra $\tT(\tT_+(V))$ in the previous theorem is called the \emph{double tensor Hopf algebra on $V$}. The grading and the connectedness of the Hopf algebra allow us to prove the existence of its antipode $S$ from the so-called \textit{Bogoliubov formula}\footnote{This terminology is in correspondence with Bogoliubov-Shirkov’s recursive formula for renormalization, as in \cite{Bogoliubov1959}.}
    \begin{equation}
        0 = S(w) + w + m_{\tT(\tT_+(V))}\circ (S\otimes\id)\circ \overline{\Delta}(w),
    \end{equation}
    for any $w\in \tT_+(V)$,  where $m_{\tT(\tT_+(V))}$ stands for the multiplication on $\tT(\tT_+(V))$ and $\overline{\Delta}$ stands for the \textit{reduced coproduct}, i.e.~$\overline{\Delta}(w) := \Delta(w) - \uno\otimes w - w\otimes \uno$, for any $w\in \tT_+(\tT_+(V))$.
\end{remark}

\subsection{Shuffle algebra structure on $\operatorname{Lin}(\tT(\tT_+(V)),\mathbb{C})$} One fundamental property of the coproduct on the double tensor Hopf algebra $\tT(\tT_+(V))$ is that it can be split into two half-coproducts. More precisely, given a word $w=a_1\cdots a_n\in V^{\otimes n}$, we consider the maps
\begin{eqnarray*}
    \Delta_\prec(w) &=& \sum_{1\in A\subseteq [[n]} w_A\otimes w^{(A)},
    \\ 
    \Delta_\succ(w) &=& \sum_{1\not\in A\subseteq [[n]} w_A\otimes w^{(A)}.
\end{eqnarray*}
Clearly, we have that $\Delta(w) = \Delta_\prec + \Delta_\succ$. The previous two maps are extended to $\tT_+(\tT_+(V))$ by declaring
\begin{eqnarray*}
    \Delta_\prec(w_1|w_2|\cdots | w_n) &=& \Delta_\prec(w_1)\Delta(w_2)\cdots \Delta(w_n),\\
    \Delta_\succ(w_1|w_2|\cdots | w_n) &=& \Delta_\succ(w_1)\Delta(w_2)\cdots \Delta(w_n),
\end{eqnarray*}
for any non-empty words $w_1,w_2\ldots,w_n\in \tT_+(V)$. 

\

It is not difficult to see that both maps are not coassociative. Instead, the authors of \cite{EFP:2015} proved that $\tT(\tT_+(V))$ together with $\Delta_\prec$ and $\Delta_\succ$ satisfy the definition of a \textit{counital unshuffle bialgebra} (also known as \textit{codendriform coalgebra}, see \cite[Thm. 5]{EFP:2015} and \cite{Foissy07} for a precise definition). This property translates to the dual space $\operatorname{Lin}(\tT(\tT_+(V)),\mathbb{C})$ as follows: for the coassociative coproduct $\Delta$ as well as $\Delta_\prec$ and $\Delta_\succ$, which from now we will call \textit{left and right half-unshuffle coproducts}, respectively, we consider the dual maps
\begin{eqnarray*}
    f*g &=& m_\mathbb{C}\circ(f\otimes g)\circ \Delta,\\
    f\prec g &=& m_\mathbb{C}\circ(f\otimes g)\circ \Delta_\prec,\\
    f\succ g &=& m_\mathbb{C}\circ(f\otimes g)\circ \Delta_\succ,
\end{eqnarray*}
for any $f,g\in \operatorname{Lin}(\tT_+(\tT_+(V)),\mathbb{C})$, where $m_\mathbb{C}$ stands for the associative product on $\mathbb{C}$. We call $\prec$ and $\succ$ the \textit{left and right half-shuffle product}, respectively. Neither of them is associative; nevertheless, they satisfy the \textit{shuffle identities}:
\begin{eqnarray}
\nonumber  (f\prec g)\prec h &=& f\prec(g*h),
    \\ (f\succ g)\prec h &=& f\succ (g\prec h), \label{eq:shuffle2}   
    \\ (f*g)\succ h &=& f\succ (g\succ h), \nonumber   
\end{eqnarray}
for any $f,g,h\in \operatorname{Lin}(\tT_+(\tT_+(V)),\mathbb{C})$. In general, any vector space $D$ together with two non-associative products $\prec$ and $\succ$ satisfying the shuffle identities with $* = \prec+\succ$ is called a \textit{shuffle algebra} (also known as \textit{dendriform algebra}, see \cite{Loday2001}). By setting $f\prec \epsilon = f = \epsilon\succ f$ and $\epsilon\prec f = 0 = f\succ f$, where $\epsilon$ is the counit on $\tT(\tT_+(V))$, we have:

\begin{theorem}[\cite{EFP:2015}]
    $\big(\!\operatorname{Lin}(\tT(\tT_+(V)),\mathbb{C}\big)$ is a unital shuffle algebra.
\end{theorem}

From the general theory of Hopf algebras, the exponential convolution defines a bijection between two special subsets in $\operatorname{Lin}(\tT(\tT_+(V)),\mathbb{C})$. More precisely, recall that a linear functional $\Phi$ on $\tT(\tT_+(V))$ is a \textit{character} if it is a unital algebra morphism. On the other hand, a linear functional $\alpha$ on $\tT(\tT_+(V))$ is an \textit{infinitesimal character} if $\alpha(\uno)=0$ and $\alpha(w_1|w_2)=0$ for any $w_1,w_2\in \tT_+(V)$. We denote $G$ and $\mathfrak{g}$ as the sets of characters and infinitesimal characters, respectively. The coassociativity of $\Delta$ implies that $G$ is a group with respect to the convolution product $*$. Besides, one can also show that $\mathfrak{g}$ is a Lie algebra with respect to the Lie bracket $[\alpha,\gamma]:= \alpha*\gamma-\gamma*\alpha$. With these notions, one has that the exponential convolution
\begin{equation}
    \mathfrak{g}\ni \alpha\mapsto \exp^*(\alpha) = \sum_{n\geq 0} \frac{\alpha^{* n}}{n!}\in G
\end{equation}
defines a bijection between $\mathfrak{g}$ and $G$. Furthermore, the half-shuffle products motivate the exponential-type maps
\begin{equation}
    \mathcal{E}_\prec(\alpha) = \sum_{n\geq0} \alpha^{\prec n}\quad\mbox{ and }\quad \E_\succ(\alpha) = \sum_{n\geq0} \alpha^{\succ n}, 
\end{equation}
for any $\alpha\in \mathfrak{g}$, where $\alpha^{\prec 0} = \alpha^{\succ 0}= \epsilon$, $\alpha^{\prec (n+1)} = \alpha\prec \alpha^{\prec n}$ for any $n\geq0$, and analogously, $\alpha^{\succ (n+1)} = \alpha^{\succ n}\succ \alpha$. Following the nomenclature, we call $\E_\prec(\alpha)$ and $\E_\succ(\alpha)$ the \textit{left and right half-shuffle exponentials of $\alpha$}, respectively. It turns out that the half-shuffle exponential maps also provide bijections between $\mathfrak{g}$ and $G$.

\begin{theorem}[\cite{EFP:2015,EFP18}]
    \label{thm:mainEFP}
    For a character $\Phi\in G$, there exists a unique triple of infinitesimal characters $(\kappa,\beta,\rho)\in \mathfrak{g}^3$ such that 
    \begin{equation}
    \label{eq:Exponentials}
        \Phi = \E_\prec(\kappa) = \E_\succ(\beta) = \exp^*(\rho).
    \end{equation}
    The infinitesimal characters $\kappa$ and $\beta$ are the unique solutions of the fixed point equations
    \begin{equation}
    \label{eq:fixedpointeqs}
        \Phi = \epsilon + \kappa\prec\Phi\quad\mbox{ and }\quad \Phi = \epsilon+\Phi\succ \beta,
    \end{equation}
    respectively. Conversely, $\E_\prec(\alpha), \E_\succ(\beta)$ and $\exp^*(\alpha)$ are characters, for any $\alpha\in\mathfrak{g}$.
\end{theorem}

The previous theorem says, in particular, that $\E_\prec(\alpha)$ is invertible with respect to the convolution product $*$, for any $\alpha\in\mathfrak{g}$. Even more, the following lemma exhibits a relation between both half-shuffle exponentials through the inverse with respect to $*$.

\begin{lemma}[{\cite[Lem. 2]{EFP:2015}}]
\label{lem:EFP15}
    For any $\alpha\in\mathfrak{g}$, we have that $\E_\prec(\alpha)^{*-1} = \E_\succ(-\alpha)$.
\end{lemma}

\section{Algebraic approach to non-commutative probability}
\label{sec:ncp}

The objective of this section is to explain how the double tensor Hopf algebra construction provides a useful framework for non-commutative probability, in particular for non-commutative cumulants. We begin by giving the combinatorial definition of cumulants for the free, Boolean and monotone case. Then, we exhibit how cumulants can be considered as linear functionals on a double tensor Hopf algebra, and indicate how shuffle-algebraic relations effectively describe relations between moments and cumulants.

\subsection{Cumulants in non-commutative probability} Recall that a pair $(\A,\varphi)$ is a \textit{non-commutative probability space} if $\A$ is a unital associative algebra over $\mathbb{C}$ and $\varphi:\A\to\mathbb{C}$ is a linear functional such that $\varphi(1_\A) = 1$, where $1_\A$ is the unit of $\A$. Elements $a\in \A$ are called random variables, and $\varphi(a)$ is called the moment of $a$ with respect to $\varphi$.

\

A non-commutative analogue of independence may be considered algebraically as a rule to compute the joint distribution of independent random variables from their marginal distributions. Muraki in \cite{Muraki} proved that there are only five such notions of independence that satisfy certain natural conditions. For each notion of independence, a non-commutative probability theory can be constructed --  one can define analogous notions of convolutions,  central limit theorems and cumulants.

\

The notions of cumulants for \textit{free, Boolean}, and \textit{monotone independence} can be introduced by the so-called moment-cumulant formulas. Before stating the next definition, we fix the following notation: given a family $\{f_n:\calA^n\to\mathbb{C}\}_{n\geq1}$ of multilinear functionals, elements $a_1,\ldots,a_n\in \A$ and $\pi\in \NC(n)$, we write 
\begin{equation}
    f_\pi(a_1,\ldots,a_n) := \prod_{B\in \pi} f_{|B|}(a_1,\ldots,a_n|B),
\end{equation}
where $f_{|B|}(a_1,\ldots,a_n|B) := f_{|B|}(a_{i_1},\ldots,a_{i_\ell})$ if $B = \{i_1<\cdots< i_\ell\}$. The linear functional $\varphi$ provides an example of a family of multilinear functionals $\{\varphi_n:\A^{n}\to\mathbb{C}\}_{n\geq1}$ by writing $\varphi_n(a_1,\ldots,a_n) := \varphi(a_1\cdot_\A\ldots\cdot_\A a_n),$ where $\cdot_\A$ stands for the associative product of the algebra $\A$.
\begin{definition}
The \textit{free} (\cite{Spe94}), \textit{Boolean} (\cite{SpW}), \textit{and monotone} (\cite{HS11}) \textit{functionals cumulants} form, respectively, the families of multilinear functionals $\{k_n:\calA^n\to \mathbb{C}\}_{n\geq1}$, $\{b_n:\calA^n\to \mathbb{C}\}_{n\geq1}$, and $\{h_n:\calA^n\to \mathbb{C}\}_{n\geq1}$, implicitly defined by the equations
\begin{eqnarray}
\varphi_n(a_1,\ldots, a_n) &=& \sum_{\pi\in\NC(n)} k_{\pi}(a_1,\ldots,a_n), \label{eq:FreeMC}
\\\varphi_n(a_1,\ldots, a_n) &=& \sum_{\pi\in\mathsf{NCInt}(n)}  b_{\pi}(a_1,\ldots,a_n),\label{eq:BoolMC}
\\\varphi_n(a_1,\ldots, a_n) &=& \sum_{\pi\in\NC(n)} \frac{1}{f_\pi!} h_{\pi}(a_1,\ldots,a_n),\label{eq:MonMC}
\end{eqnarray}
for any $n\geq1$ and $a_1,\ldots,a_n\in \calA$, where $f_\pi$ stands for the forest of nestings of $\pi$. 
\end{definition}

\begin{remark}
It is not difficult to see that the previous combinatorial formulas indeed recursively define the free, Boolean and monotone cumulants. Actually, for the free and Boolean cases, we can take advantage of the Möbius inversion formulas in the posets $\NC(n)$ and $\mathsf{NCInt}(n)$, respectively, in order to obtain the corresponding cumulant-moment relations. However, this approach cannot be directly followed in the monotone case since the coefficients $\frac{1}{f_\pi!}$ in \eqref{eq:MonMC} do not satisfy the multiplicativity condition to apply Möbius inversion. 
\end{remark}

\subsection{Shuffle-algebraic framework for non-commutative cumulants} 
In \cite{EFP:2015,EFP18}, the authors established a group-theoretical framework for the moment-cumulants relations based on a double tensor Hopf algebra described in \cref{sec:doubletensor}. More precisely, given a non-commutative probability space $(\A,\varphi)$, consider $\tT(\tT_+(\A))$ the double tensor Hopf algebra on $\A$. Then, we extend the linear functional $\varphi:\A\to\mathbb{C}$ to a character $\Phi:\tT(\tT_+(\A))\to\mathbb{C}$ by setting
\begin{equation}
    \label{eq:character}
    \Phi(w) = \varphi_n(a_1,\ldots, a_n),
\end{equation}
for any word $w=a_1\cdots a_n\in \A^{\otimes n},$ and extending multiplicatively and linearly for all the elements of $\tT(\tT_+(\A)).$ With these notions, the connection with non-commutative probability is stated in the following theorem.

\begin{theorem}[\cite{EFP:2015,EFP18}]
    \label{thm:link}
    Let $(\A,\varphi)$ be a non-commutative probability space and let $\Phi : \tT(\tT_+(\A))\to \mathbb{C}$ the extension of $\varphi$ to a character as described above. Let $(\kappa,\beta,\rho)$ be the unique triple of infinitesimal characters given in \cref{thm:mainEFP}. Then $\kappa,\beta$ and $\rho$ correspond with the free, Boolean and monotone cumulants, respectively:
\begin{equation}
    \kappa(w) = k_n(a_1,\ldots,a_n),\qquad \beta(w) = b_n(a_1,\ldots,a_n),\qquad \rho(w) = h_n(a_1,\ldots,a_n),
\end{equation}
for any $w=a_1\cdots a_n\in \A^{\otimes n}$.
\end{theorem}

The strategy to prove the previous theorem is to observe that the three equations in \eqref{eq:Exponentials} recover the moment-cumulant relations \eqref{eq:FreeMC}, \eqref{eq:BoolMC} and \eqref{eq:MonMC} when evaluating on a word $w\in \tT_+(\A).$ More precisely, we have for any $w=a_1\cdots a_n \in \tT_+(\calA) $:

\begin{eqnarray*}
    \E_\prec(\kappa)(w) &=& \sum_{\pi\in\NC(n)} \prod_{B\in\pi}\kappa(w_B), 
\\\E_\succ(\beta)(w) &=& \sum_{\pi\in\mathsf{NCInt}(n)} \prod_{B\in\pi}\beta(w_B), 
\\\exp^*(\rho)(w) &=& \sum_{\pi\in\NC(n)} \frac{1}{f_\pi!}\prod_{B\in\pi}\rho(w_B). 
\end{eqnarray*}

\cref{thm:link} provides a shuffle-algebraic point of view for moments and non-commutative cumulants, which have been recently studied to understand in a unified way several notions in non-commutative probability theory such as additive convolutions (\cite{EFP18}), conditional free cumulants (\cite{EFP20}), and recently, solving the open problems of finding combinatorial relations to write monotone cumulants in terms of moments (\cite{AC22}), and free and Boolean cumulants (\cite{CEFPP22}).

\

Let us recall one context where the shuffle-algebraic framework for non-commutative probability is applied. Recall that in the classical case, if $X$ is a classical random variable with finite moments of all orders, there exists a particular family $\{W_n(x)\}_{n\geq0}$ of polynomials associated to $X$ closely related with the classical cumulants of $X$, called \textit{Wick polynomials}, which are determined by the properties $W_0(x)=1$ and
$$\mathbb{E}(W_n(X)) =0,\qquad \frac{\mathrm{d}}{\mathrm{d}x}W_n(x) = nW_{n-1}(x),\quad\mbox{ for any }\,n>0.$$
In \cite{Ansh04}, the author introduced the free counterpart of the Wick polynomials which satisfy analogous properties that of classical Wick polynomials. Roughly speaking, the free Wick polynomials appear by considering the family of free cumulants of a random variable instead of its classical cumulants. This situation is described in Hopf-algebraic terms as follows. Let $(\A,\varphi)$ be a non-commutative probability space, and consider the double tensor Hopf algebra $\tT(\tT_+(\A))$ as well as the character $\Phi$ on $\tT(\tT_+(\A))$ extending $\varphi$. Since $\Phi$ is a character, then $\Phi^{*-1}$ exists and we can define the \emph{free Wick map} $W:\tT(\tT_+(\A))\to \tT(\tT_+(\A))$ by the recipe
\begin{equation}
\label{eq:Wick}
    W := (\id\otimes \Phi^{*-1})\circ \Delta.
\end{equation}
The map $W$ was introduced in \cite{EFPTZ}, where the authors proved that the evaluation $W(a_1\cdots a_n)$ coincides with the free Wick polynomials introduced in \cite{Ansh04}. This can be stated in the following result.

\begin{proposition}[{\cite[Prop. 3.12]{Ansh04}}]
    Let $(\A,\varphi)$ be a non-commutative probability space with family of free cumulants $\{k_n\}_{n\geq1}$. Consider the double tensor Hopf algebra $\tT(\tT_+(\A))$ and the character $\Phi$ on $\tT(\tT_+(\A))$ extending $\varphi$. Also, consider the free Wick map $W$ defined in \eqref{eq:Wick}. Then for a word $w = a_1\cdots a_n\in \tT_+(\A)$, we have that
    \begin{equation}
    \label{eq:FreeWick}
     W(a_1\cdots a_n) = \sum_{\substack{B\subseteq [n]\\ B = \{i_1<\cdots < i_s\}}}  a_{i_1}\cdots a_{i_s}\sum_{\substack{\pi\in \mathsf{NCInt}([n]\backslash B)\\ \{B\}\cup \pi\in \NC(n)}} (-1)^{|\pi|} k_\pi\big(a_1,\ldots,a_n\,\big|\,[n]\backslash B\big),
    \end{equation}
    for any $n\geq1$ and  $a_1,\ldots,a_n\in \A$.
\end{proposition}
\begin{proof}
    Let $\kappa$ be the infinitesimal character such that $\Phi = \E_\prec(\kappa)$. By \cref{thm:link}, $\kappa$ identifies with the free cumulants when evaluating on a word $w\in \tT_+(\A)$. Also, \cref{lem:EFP15} implies that
     $\Phi^{*-1} = \E_\prec(\kappa)^{*-1} = \E_\succ(-\kappa)$. In the language of non-commutative probability, this means that $-\kappa$ is the infinitesimal character extending the Boolean cumulant functionals of $\Phi^{*-1}$. Then, we have 
    \begin{equation}
    \label{eq:PhiInvBool}
        \Phi^{*-1}(w) = \sum_{\pi\in \mathsf{NCInt}(n)} \prod_{B\in \pi} \big(\!-\kappa(w_B)\big) = \sum_{\pi\in\mathsf{NCInt}(n)} (-1)^{|\pi|}k_\pi(a_1,\ldots,a_n),
    \end{equation}
    for $n\geq 1$ and any word $w=a_1\cdots a_n \in \A^{\otimes n}$. On the other hand, the definition of $W$ and  the fact that $\Phi^{*-1}$ is multiplicative, we have for a word $w=a_1\cdots a_n$:
    \begin{eqnarray*}
        W(w) &=& (\id\otimes \Phi^{*-1})\circ \Delta(w)
        \\ &=& \sum_{B\subseteq[n]} w_B\prod_{i=1}^s \Phi^{*-1}(w_{K_i}),
        \end{eqnarray*}
    where $K(B) = \{K_1,\ldots,K_s\}$ is the decomposition of $[n]\setminus B$ into its connected components. Notice that we can use \eqref{eq:PhiInvBool} on each $\Phi^{*-1}(w_{K_i})$ in order to write each factor in terms of free cumulants. Furthermore, given $B\subseteq[n]$, observe that the fact that $K_i$ and $K_{i+1}$ are not consecutive intervals, for every $1\leq i<s$, implies that the conditions $\pi\in \mathsf{NCInt}([n]\backslash B)$ and $\{B\}\cup \pi\in\NC(n)$ mean that $\pi = \pi_1\sqcup \pi_2\sqcup\cdots\sqcup \pi_s$, where $\pi_i \in \mathsf{NCInt}(K_i)$ for each $1\leq i\leq s$. Therefore, we get
    \begin{eqnarray*}
        W(w) &=&  \sum_{B\subseteq[n]} w_B\prod_{i=1}^s \Phi^{*-1}(w_{K_i})
        \\ &=&  \sum_{B\subseteq[n]} w_B\prod_{i=1}^s \sum_{\pi_i\in\mathsf{NCInt}(K_i)} (-1)^{|\pi_i|} k_{\pi_i}(a_1,\ldots,a_n|K_i)
        \\ &=& \sum_{B\subseteq[n]} w_B \sum_{\substack{\pi\in\mathsf{NCInt}([n]\backslash B)\\ \{B\}\cup \pi\in\NC(n)}} (-1)^{|\pi|} k_{\pi}\big(a_1,\ldots,a_n\,\big|\,[n]\backslash V\big),
    \end{eqnarray*}
    as we wanted to show.
\end{proof}

\section{The antipode of $\mathsf{S}(\mathsf{T}_+(V))$ and forest formulas}
\label{sec:antipodeSS}

As a preliminary step towards finding the antipode in a double tensor Hopf algebra $\tT(\tT_+(V))$, we begin by addressing the same issue in its commutative counterpart, the symmetric algebra of the augmented tensor algebra, $\sS(\tT_+(V))$. The goal of this section is to understand the antipode formula using a general method that was previously developed in \cite{Menous2018}. This method involves calculating a formula for the iterated coproduct using a specific class of decorated non-planar rooted trees, which naturally incorporates Schröder trees into the antipode formula.

\subsection{Right-handed polynomial Hopf algebras and forest formulas} In the work \cite{Menous2018}, the authors introduce a method to produce combinatorial tree-indexed formulas for the iterated coproduct and the antipode for a particular class of Hopf algebras that are enveloping algebras of pre-Lie algebras (see \cite[Chap. 6]{CaP22}). The precise definition of this class of Hopf algebras can be stated as follows.

\begin{definition}
    Let $V$ be a vector space. A \textit{right-handed polynomial  Hopf algebra} is a polynomial algebra $\sS(V)$ together with a coproduct $\delta$ that makes $\sS(V)$ a Hopf algebra and such that $\delta$ satisfies
\begin{equation}
    \overline{\delta}(V) \subset V \otimes \sS(V),
\end{equation}
where $\overline{\delta}$ is the reduced coproduct $\overline{\delta}(x) = \delta(x) - x \otimes 1_{\sS(V)} - 1_{\sS(V)}\otimes x$.
\end{definition}

\begin{example}
\label{ex:decoration1}
Let $V$ be a finite-dimensional vector space, and consider $\sS(\tT_+(V))$ the symmetric algebra of the tensor algebra. Notice that $\tT_+(V)$ has a countable basis given by the set $\mathcal{W}$ of non-empty words with letters in $V$. 
By symmetrizing the expression for the coproduct in $\tT(\tT_+(V))$, the reduced coproduct in $\sS(\tT_+(V))$ on a word $w\in\mathcal{W}$ of length $|w|=n$, denoted by $\overline{\Delta}_\sS$, is given by
\begin{equation}
    \label{eq:STVCop}
    \overline{\Delta}_\sS(w) = \sum_{\emptyset\neq A \subsetneq [n]} w_A \otimes \tilde{w}^{(A)},
\end{equation}
where if $K(A) = \{K_1,\ldots,K_s\}$, then $\tilde{w}^{(A)}$ denotes the commutative monomial in $\sS(\tT_+(V))$ given by the polynomial product $w_{K_1}\cdot \ldots \cdot w_{K_s}$. In particular, we have that $\sS(\tT_+(V))$ is a right-handed polynomial Hopf algebra. 
\end{example}

\par We now describe the method developed in \cite{Menous2018} to find an antipode formula for right-handed polynomial Hopf algebras. First, assume that $\sS(V)$ is a conilpotent right-handed polynomial Hopf algebra and that that $V$ has a countable basis $\mathcal{B} = \{b_i\}_{i\geq1}$. Then, the reduced coproduct of an element $b_i\in \mathcal{B}$ can be expanded as
\begin{equation}
    \label{eq:generalforest0}
    \overline\delta (b_i)=\sum_{i_0, I\not=\emptyset} \lambda^{i ; i_0}_{I} b_{i_0}\otimes b_{I},
\end{equation}
for coefficients $\lambda^{i;i_0}_I\in \mathbb{K}$, where the above sum runs over the integers $i_0\geq1$ and non-empty multisets $I\subseteq \mathbb{N}$, where $b_I := b_{i_1}\cdot  b_{i_2}\cdot\ldots\cdot b_{i_s}$ is the monomial given by the (commutative) polynomial product of the basis elements indexed by $I=\{i_1,i_2,\ldots,i_s\}$. Notice that the $\lambda^{i;i_0}_I$-coefficients completely determine the coproduct, its action on products of elements of $\mathcal{B}$, as well as the action of the iterated reduced coproducts $\overline{\delta}^{[k]}$. 

\

 A key observation in \cite{Menous2018} is that the summation in \eqref{eq:generalforest0} can be indexed by non-planar decorated corollas in the following way:

\begin{equation}
\label{eq:generalforest}
\overline\delta (b_i)=\sum \lambda\Big(\corollaa \Big) b_{i_0}\otimes b_{i_1}\cdots b_{i_k}.
\end{equation}
The decoration of the corollas is given by a pair of positive integers $(i;i_0)$ associated to the root, and leaves are decorated by positive integers $i_1,\ldots,i_k$, with $k\geq1$. The commutativity of the elements in $\sS(V)$ is reflected by the fact that we are considering non-planar trees. Furthermore, the expression \eqref{eq:generalforest} suggests that iterations of the reduced coproduct can be encoded in a family of more general non-planar decorated trees. To this end, we describe the following notions associated to decorated trees.

\begin{definition}[\cite{Menous2018}]\label{def:tree}
Let $T$ be a non-planar finite rooted tree whose internal vertices are decorated by pairs $(p_1;p_2)$ of positive integers, and whose leaves are decorated by positive integers. In the case of $T$ being a  single-vertex tree, the vertex is considered as a leaf.
\begin{enumerate}[i)]
    \item  For any $x\in \operatorname{Int}(T)$ internal vertex of $T$, we denote $d(x)=(d_1(x);d_2(x))$ its decoration. If $x$ is a leaf of $T$, we denote its decoration $d(x)=d_1(x)=d_2(x)$. 
    \item If the root of $T$ is decorated by $i$ or $(i;i_0)$, we say that $T$ is \textit{associated to $b_i\in \mathcal{B}$}. The set of decorated trees associated to $b_i$ is denoted by $\mathcal{T}_{i}$.
    \item For a given pair of positive integers $(i;i_0)$, we denote by $B_+^{(i;i_0)}(T_1,\ldots, T_s)$ the decorated tree obtained by adding a common root decorated by $(i;i_0)$ to the decorated trees $T_1,\ldots, T_s$. {If $T = B_+^{(i;i_0)}(T_1,\ldots, T_s)$, we denote by $B_-(T)$ the multiset of trees $\{T_1,\ldots, T_s$\}.}
    \item Let $F$ be a multiset of decorated trees given by
    \begin{equation}
        \label{eq:multiforest}
        F=\{T_{1,i_1}^{k_{1,1}},\dots,T_{s_1,i_1}^{k_{s_1,1}}\}\cup \dots \cup \{T_{1,i_p}^{k_{1,p}},\dots,T_{s_p,i_p}^{k_{s_p,p}}\},
    \end{equation}
    where: \begin{itemize}
        \item the tree $T_{j,i_q}$ is associated to $i_q$, for every $1 \leq q \leq p$; 
        \item the trees $T_{j,i_q}$ are all distinct, for every $1\leq j\leq s_q$;
        \item the notation $T_{j,i_q}^{k_{j,q}}$ means that the tree $T_{j,i_q}$ appears with multiplicity $k_{j,q}$ in the multiset $F$, for every $1 \leq j \leq s_q$ and $1 \leq q \leq p$.
    \end{itemize}  
    \vspace{.1in}
    The \textit{symmetry coefficient} of $F$, denoted by $\operatorname{sym}(F)$, is then given by $$\operatorname{sym}(F):=\prod_{j=1}^p {{k_{1,j}+\cdots+k_{s_j,j}}\choose{k_{1,j},\dots,k_{s_j,j}}}.$$
    \item We define the coefficient $\lambda(T)$ as follows: if $\bullet_i$ stands for the single-vertex tree with decoration $i$, then $\lambda(\bullet_i):=1$. In general, if $T=B_+^{(i;i_0)} (T_1,\ldots, T_s)$, then 
\begin{equation}
    \label{def:lambdaT}
    \lambda(T):=\lambda^{i;i_0}_{i_1,\dots,i_s} \operatorname{sym}(F)\lambda(T_1)\cdots\lambda(T_s),
\end{equation}
when $T_1,\dots, T_s$ are trees respectively associated to $b_{i_1},\ldots,b_{i_s}$, and $F$ is the multiset $\{T_1,\ldots,T_s\}$. 

\end{enumerate}
\end{definition}

It was observed in \cite{Menous2018} that 
a way to describe all the tensors of length $k$, that can be obtained in the $k$-fold iterated reduced coproduct, can be obtained by using the notion of $k$-linearization of a poset. 

\begin{definition}
    Let $P$ be a finite poset. We say that $f:P\to [k]$ is a  \textit{$k$-linearization of $P$} if $f$ is a surjective, strictly order-preserving map $f:P\to[k]$. We denote by  $k\operatorname{-lin}( P)$, the set of $k$-linearizations of $P$.
\end{definition}

Recall that any tree can be regarded as a poset where the root is a minimal element. In particular, if $T$ is a decorated tree with a given decoration $d=(d_1;d_2)$ and $f$ is a $k$-linearization of $T$, we can write
$$C(f):=\left(\prod_{x_1\in f^{-1}(1)} b_{d_2(x_1)}\right) \otimes \cdots \otimes \left(\prod_{x_k\in f^{-1}(k)} b_{d_2(x_k)}\right),$$
where we recall that the product in the above expression is the commutative polynomial product.

\begin{theorem}[{\cite[Lem. 12]{Menous2018}, \cite{celestino2022forest}}]
\label{thm:forestform0}
Let $\sS(V)$ be a right-handed polynomial Hopf algebra with coproduct $\delta$, such that $V$ has a countable basis $\mathcal{B} = \{b_i\}_{i\geq1}$. Then, for any $b_i \in \mathcal{B}$, we have for the action of the $k$-fold iterated reduced coproduct:
\begin{equation}
\label{eq:forestformula0}
\overline\delta^{[k]}(b_i)=\sum\limits_{T\in \mathcal{T}_i}\sum_{f\in k\operatorname{-lin} (T)}\lambda(T)C(f).
\end{equation}
\end{theorem}

The formula in \eqref{eq:forestformula0} is called the \textit{forest formula for iterated coproducts}. Recalling \textit{Takeuchi's formula} \cite{takeuchi1971free} for the antipode $S_H$ in a Hopf algebra $H$ with product $m_H$ and reduced coproduct $\overline{\Delta}_H$:
\begin{equation}
\label{eq:takeuchi}
    S_H(b) = \sum_{k\geq1} (-1)^k m_H^{[k]}\circ \overline{\Delta}_H^{[k]}(b),\quad\mbox{ for all $b\in H$},
\end{equation}
we have the following cancellation-free forest formula for the antipode of $\sS(V)$.

\begin{theorem}[{\cite[Thm. 8]{Menous2018}}]
\label{thm:forestAntipode}
With the notation of the \cref{thm:forestform0}, the evaluation of the antipode $S_V$ of the right-handed polynomial Hopf algebra $\sS(V)$ on an element $b_i\in \mathcal{B}$ is given by the cancellation-free formula:
\begin{equation}
    \label{eq:forestAntipode}
S_V(b_i)=\sum\limits_{T\in \mathcal{T}_i}(-1)^{|T|}\lambda(T)\tau(T),
\end{equation}
where $\tau(T)= \prod\limits_{x\in \operatorname{Vert}(T)} b_{d_2(x)},$ for any $T\in \mathcal{T}_i$.
\end{theorem}

\subsection{Application to the case of $\sS(\tT_+(V))$} We are now interested in applying the above machinery for the case of the symmetric algebra of the augmented tensor algebra of a finite-dimensional vector space. From \cref{ex:decoration1}, we have that $\sS(\tT_+(V))$ is a right-handed polynomial Hopf algebra such that the set of non-empty words $\mathcal{W}$ on letters in $V$ is a countable basis of $\tT_+(V)$. For notational convenience, we give an enumeration $\mathcal{W} = \{w_i\}_{i\geq1}$. 

\begin{remark}
    The formula for the antipode in $\sS(\tT_+(V))$, in the general case in which $V$ is not finite-dimensional, will also follow from the finite-dimensional case. This can be noticed by the fact that, for a word $w=a_1\cdots a_n\in V^{\otimes n}$, the coproduct $\Delta(w)$ is a linear combination of elements of the form $w'\otimes \mathfrak{w}$, where $w'$ is a subword of $w$ and $\mathfrak{w}$ is a product of subwords of $w$. Therefore, we can obtain $S(w)$ by computing the formula for the antipode in $\sS(\tT_+(V_w))$, where $V_w$ is the finite-dimensional vector space generated by $\{a_1,\ldots,a_n\}$.
\end{remark}

We begin by writing the expression for $\overline{\Delta}_\sS$ in \eqref{eq:STVCop} as a sum indexed by decorated corollas as in \eqref{eq:generalforest}. To this end, take $A\subsetneq[n]$ a non-empty subset, i.e.~$K(A)=\{K_1,\ldots,K_s\}$ is non-empty and does not contain $[n]$. Now, consider a corolla $C_0$ with $|A|+1$ leaves, and for any $K_i\in K(A)$, consider a corolla $C_i$ with $|K_i|+1$ leaves. Now we label the sectors of $C_A$ with the elements of $A$ increasingly from left to right, i.e.~by composing the natural labelling of the sectors of $C_0$ with the unique increasing bijection from $A$ to $|A|$. Moreover, we decorate $C_i$ with the elements of $K_i$ in the same way. Notice that the definition of $K(A)$ implies that its elements are non-consecutive intervals. This implies that we can graft $C_1,\ldots,C_s$ to different leaves of $C_0$ in order to obtain a Schröder tree $t_A\in \ST(n)$, in such a way that the labellings of the sectors of the corollas coincide with the natural labelling of $t_A$.

\

Alternatively, given $\emptyset\neq A\subsetneq[n]$, we can describe $t_A\in \ST(n)$ as follows: consider the non-crossing partition $\sigma_A = \{A\}\cup K(A)\in \NC(n)$. Then $t_A$ is the unique Schröder tree with $n+1$ leaves such that $\pi(t_A) = \sigma_A$ and the block $A\in \sigma_A$ is determined by the labels of the sectors adjacent to the root of $t_A$, given by the natural labelling of $t_A$. For instance, for $n=10$ and $A = \{ 3,4,7,9\}$, then we have

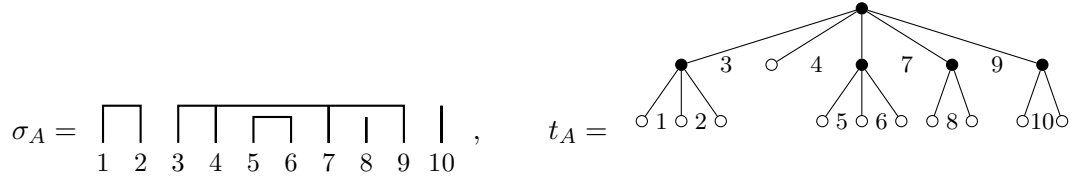
\begin{figure}[H]
\centering
\begin{subfigure}{.4\textwidth}
  \centering
 $\sigma_A =$  \begin{tikzpicture}[baseline={([yshift=0.3ex]current bounding box.center)},thick,font=\small]
     \path (0,0) 		node (a) {1}
           	(0.5,0) 	node (b) {2}
           	(1,0) 		node (c) {3}
           	(1.5,0) 	node (d) {4}
           	(2,0) 		node (e) {5}
           	(2.5,0) 	node (f)  {6}
           	(3,0) 		node (g) {7}
           	(3.5,0) 	node (h) {8}
           	(4,0) 		node (i) {9}
           	(4.5,0) 	node (j)  {10};
     \draw (a) -- +(0,0.75) -| (b);
     \draw (c) -- +(0,0.75) -| (i);
     \draw (d) -- +(0,0.75) -| (d);
     \draw (e) -- +(0,0.60) -| (f);
     \draw (g) -- +(0,0.75) -| (g);
     \draw (h) -- +(0,0.60) -| (h);
     \draw (j) -- +(0,0.75) -| (j);
   \end{tikzpicture}
\end{subfigure},$\qquad$ $t_A =$ 
\begin{subfigure}{.4\textwidth}
  \centering
  \begin{tikzpicture}[scale=1.5,font=\footnotesize,
 level 1/.style={level distance=5mm,sibling distance=8mm},
 level 2/.style={level distance=5mm,sibling distance=3.5mm},
 level 3/.style={level distance=6mm,sibling distance=5mm},
 level 4/.style={level distance=10mm,sibling distance=5mm},
]
\node(0)[solid node,label=above:{}]{} 
child{node(1)[solid node]{}
    child{[black] node(11)[hollow node]{}}
    child{[black] node(12)[hollow node]{}}
    child{[black] node(13)[hollow node]{}}
}    
child{node(2)[hollow node]{}
}
child{node(3)[solid node]{}
    child{[black] node(31)[hollow node]{}}
    child{[black] node(32)[hollow node]{}}
    child{[black] node(33)[hollow node]{}}
}
child{node(4)[solid node]{}
child{[black] node(41)[hollow node, ]{}}
child{[black] node(42)[hollow node,]{}
}
edge from parent node[right]{}
}
child{node(5)[solid node]{}
    child{[black] node(51)[hollow node]{}}
    child{[black] node(52)[hollow node]{}}
};

\path (1) -- node {$3$} (2);
\path (2) -- node {$4$} (3);
\path (3) -- node {$7$} (4);
\path (4) -- node {$9$} (5);

\path (11) -- node (H) {${1}$} (12);
\path (12) -- node {$2$} (13);
\path (31) -- node {$5$} (32);
\path (32) -- node {$6$} (33);
\path (41) -- node {$8$} (42);
\path (51) -- node {$10$} (52);

\end{tikzpicture}
\end{subfigure}%

\label{fig:DecSchTreeEx} 
\end{figure}
\noindent where it is clear that $\pi(t_A) = \sigma_A$. On the other hand, notice that the planarity of Schröder trees describes the order between the elements of $K(A)$. In other words, given a Schröder tree $t$ of \textit{height} 2, i.e.~$t$ is such that the largest path from the root to any leaf contains exactly two internal vertices, we can recover a non-empty subset $A\subsetneq [n]$ indexing \eqref{eq:STVCop}. Hence, we can state:

\begin{lemma}
\label{lem:decoration}
    Let $w\in \mathcal{W} \subset \tT_+(V)$ be a word of length $|w| = n$. Then, the reduced coproduct on $\sS(\tT(V))$ can be written as
    
\begin{equation}
    \label{ex:decoration}
    \overline{\Delta}_\sS(w) = \sum_{\substack{t\in \ST(n)\\t\mbox{ \scriptsize has height }2}} w_A\otimes \tilde{w}^{(A)},
\end{equation}
where $A$ is the block of $\pi(t)$ associated to the root of $t$, and $w^{(A)} = w_{K_1}\cdot\ldots\cdot w_{K_s}$, where $K(A) = \{K_1,\ldots,K_s\}$ is the list of blocks associated to the internal vertices different than the root.
\end{lemma}

Even more, the term in the above sum can be equivalently indexed by a decorated version of the \textit{non-planar skeleton of $t$}, denoted by $\widetilde{\sk}(t)$, which is defined to be the non-planar rooted tree generated by the $\operatorname{Int}(t)$. Notice that for the case of Schröder trees of height 2, its skeleton is precisely a corolla with exactly $i(t)-1$ leaves. In addition, the decoration of $\widetilde{\sk}(t)$ is given by labelling the root with $(i;i_0)$, where $i$ and $i_0$ are the indexes in the list $\mathcal{W}=\{w_j\}_{j\geq1}$ such that $w=w_i$ and $w_{i_0}=w_A$, respectively, where $A$ is the block of $\pi(t)$ associated to the root of $t$, and the leaves are decorated with $i_1,\ldots,i_s$, where for each $1\leq j\leq s$, $i_j$ is the index such that $w_{i_j} = w_{K_j}$, with $K(A)=\{K_1,\ldots,K_s\}$. Thus, we have shown that \eqref{ex:decoration} can be written in the form \eqref{eq:generalforest}, where the respective $\lambda^{i;i_0}_{i_1,\ldots,i_s}$-coefficients are precisely described below.


\begin{lemma}
\label{lem:auxS7}
Let $\mathcal{W}=\{w_j\}_{j\geq1}$ be the set of non-empty words on a finite-dimensional vector space $V$. For any indexes $i,i_0,i_1,\ldots,i_s\geq1$, the coefficient $\lambda^{i;i_0}_{i_1,\ldots,i_s}$ is the number of ways in which it is possible to select interval subwords $v_1,v_2,\ldots,v_{s+1}$ of $w_i$ such that: 
\begin{itemize}
    \item $v_j$ is an interval subword, for any $1\leq j\leq s+1$, i.e.~$v_j = w_{J_j}$ for some interval $J_j \subseteq [|w_i|]$;  
    \item $v_j$ is a non-empty word, for $2\leq j\leq s$; 
    \item $w_{i_0} = v_1v_2\cdots v_{s+1}$, i.e.~the concatenation of $v_1,\ldots,v_{s+1}$ is $w_{i_0}$; 
    \item there exists a permutation $\sigma\in \Sigma_s$ such that $w_i = v_1w_{i_{\sigma(1)}}v_2 w_{i_{\sigma(2)}}\cdots v_sw_{i_{\sigma(s)}}v_{s+1}$.
\end{itemize}
\end{lemma}
\begin{proof}
    Assume that there are $s+1$ interval subwords $v_1,\ldots,v_{s+1}$ of $w:=w_i = a_1\cdots a_n\in V^{\otimes n}$ such that $v_2,\ldots,v_s$ are non-empty, $w_{i_0}= v_1\cdots v_{s+1}$, and $w_i = v_1w_{i_{\sigma(1)}}v_2w_{i_{\sigma(2)}}\cdots v_sw_{i_{\sigma(s)}}v_{s+1}$ for a permutation $\sigma\in \Sigma_s$. 
    Now, for each $1\leq j\leq s+1$, define $B_j\subseteq[n]$ such that $v_j = w_{B_j}$. This produces a list of pairwise disjoint subsets $B_1,\ldots,B_{s+1}\subseteq [n]$ such that $B_2,\ldots, B_s\neq \emptyset$. We now set $A = B_1\cup \cdots \cup B_{s+1}$. The remaining indexes $[n]\backslash A$ can be grouped into $s$ non-empty pairwise disjoint and non-consecutive intervals $K_1,\ldots,K_s$. By construction, it is clear that $K(A) = \{K_1,\ldots,K_s\}$, $w_{i_0} = w_A$ and $\tilde{w}^{(A)} = w_{i_1}\cdot\ldots\cdot w_{i_s},$ so that every selection of $s+1$ internal subwords $v_1,\ldots,v_{s+1}$ as described at the beginning will produce term in the sum \eqref{eq:STVCop}. Since different selection of the internal subwords will produce different subsets $A$ that give the same term $w_{i_0}\otimes w_{i_1}\cdot\ldots\cdot w_{i_s},$ we have that $\lambda_{i_1,\ldots,i_s}^{i;i_0}$ counts the number of the desired decompositions of $w_i$, as we wanted.
    \end{proof}

Now we are ready to apply the above machinery to compute the antipode of $\sS(\tT_+(V))$ via the forest formula in \cref{thm:forestAntipode}. In order to write the forest formulas for a fixed $w_i\in\mathcal{W}$, we need to describe $\T_i$, the set of decorated trees associated to $w_i$.

\begin{proposition}
\label{prop:lambdaCoef}
    Let $w_i\in \mathcal{W} \subset \tT_+(V)$ be a non-empty word such that $|w_i|=n$ and consider $\mathcal{T}_i'$ the subset of $T\in \mathcal{T}_i$ such that $\lambda(T) \neq 0$. 
    Then there exists a surjective map $\Lambda:\ST(n)\to \mathcal{T}_i'$ such that $|\Lambda^{-1}(T)| = \lambda(T)$, for any $T\in\mathcal{T}_i'$.
\end{proposition}
\begin{proof}
First, we will give the description of the map $\Lambda$, which is given by a decorated version of the non-planar skeleton of $t$, $\widetilde{\sk}(t)$. Assume that $w:=w_i= a_1\cdots a_n$ with $a_1,\ldots,a_n\in V$. We generalize the decoration of corollas to more general trees as follows. Take a vertex $v$ of $\widetilde{\sk}(t)$ and consider it as an internal vertex of $t$. Then, let $B_v$ be the block of $\pi(t)$ associated to $v$. Then, we set $d_2(v)$ to be the index such that $w_{d_2(v)} = w_{B_v}$. For $d_1(v)$, note that $v$ is the root of a sub-Schröder tree $t'$ of $t$ consisting of $v$ and all its descendants. By setting $C_v = \bigcup_{B \in \pi(t')} B$, we set $d_1(v)$ to be the index such that $w_{d_1(v)} = w_{C_v}$. With this construction, define $\Lambda(t)$ as the decorated tree $T$ given by the non-planar skeleton of $t$ together with the previous decoration of the vertices of $\widetilde{\sk}(t)$. It is clear that if $r$ is the root of $T$, then $w_{d_1(r)} = w_i$. Furthermore, the previous construction implies that, for any $v\in \operatorname{Vert}(\widetilde{\sk}(t))$, $w_{d_2(v)}$ is a subword of $w_{d_1(v)}$. Also, the set of indexes 
$d_1(\operatorname{succ}(v)) := \{d_1(x)\,:\, x\in \operatorname{succ}(v)\}$ is such that $$\{w_{d_1(x)}\,:\,x\in \operatorname{succ}(v)\} = \{w_K\,:\,K\in K(B_v,C_v)\}.$$ 
Hence $\lambda_{d_1(\operatorname{succ}(v))}^{d_1(v);d_2(v)}\neq0$ for any $v\in \operatorname{Vert}(\widetilde{\sk}(t))$, which implies that $\lambda(T)\neq0$, so that $\Lambda(T)\in \T_i'$ and $\Lambda:\ST(n)\to \T_i'$ is well-defined.
\par Now, we prove that the map $\Lambda$ is surjective. Take $T\in \T_i'$. The proof will be done by induction on $k=|T|$, the number of vertices of $T$.  If $T$ is a single-vertex tree, i.e. $k=1$, we take $t\in\ST(n)$ to be the corolla with $n+1$ leaves, which is the unique Schröder tree such that $\Lambda(t) = T$.  Otherwise, assume that $T = B_+^{(i;i_0)}(T_1,\ldots,T_s)$, where $T_j$ is associated to $w_{i_j}$ for any $1\leq j\leq s$. Since $\lambda(T)\neq 0$, then we have that $m=\lambda^{i;i_0}_{i_1,\ldots,i_s}>0$. By \cref{lem:auxS7}, we can find $m$ non-empty subsets $A\subsetneq [n]$ such that $w_{i_0} = w_A$ and $\tilde{w}^{(A)} = w_{i_1}\cdot \ldots\cdot w_{i_s}$. Due to the possible repetitions in the subwords $w_{i_j}$, there are $\operatorname{sym}(B_-(T))$ ways to allocate the decorated trees $T_1,\ldots,T_s$ to the subwords $w_{K_1},\ldots,w_{K_s}$, where $K(A) = \{K_1,\ldots,K_s\}.$ Moreover, since $|T_j|< k$, we can find $\lambda(T_j)$ different Schröder trees $t_j\in \ST(|K_j|)$ such that their corresponding decorated skeletons are given by the same $T_j$, for each $1\leq j\leq s$. Hence, by grafting properly the Schröder trees $t_1,\ldots,t_s$ to a corolla with $|A|+1$ leaves, we can construct
$$m\operatorname{sym}(B_-(T))\lambda(T_1)\cdots \lambda(T_s) = \lambda(T)$$
Schröder trees $t\in \ST(n)$ with $k$ internal vertices such that $\Lambda(t) = T$, which completes the proof. 
\end{proof}

The main application of the previous proposition is that it allows us to rewrite the forest formulas \eqref{eq:forestformula0} and \eqref{eq:forestAntipode} in terms of Schröder trees. In particular, the second sum in \eqref{eq:forestformula0} can be indexed by $k$-linearizations of the skeleton of Schröder trees. For the purpose of writing the formulas, we introduce the next notation.

\begin{notation}
\label{not:iterated}
     Let $w \in \tT_+(V)$ be a non-empty word on $V$ such that $|w| = n$. For $t\in \ST(n)$ such that $\pi(t) = \{B_1,\ldots,B_r\}$ and $f\in k\operatorname{-lin}(\widetilde{\sk}(t))$, we define 
        $$
        \tilde{w}_t := \prod_{B\in \pi(t)} w_B = w_{B_1}\cdot \ldots \cdot w_{B_r}\in \sS(\tT_+(V))
        $$
     and
      \begin{equation}
        \tilde{c}(w,t,f) := \left(\prod_{B\in \pi_1} w_B\right)\otimes\cdots\otimes \left(\prod_{B\in\pi_k} w_{B}\right)\in \sS(\tT_+(V))^{\otimes k},
    \end{equation}
    where $\pi_j$ is the collection of blocks of $\pi(t)$ associated to the internal vertices $v$ of $t$ such that, regarded as elements of the poset $\widetilde{\sk}(t)$, $v\in f^{-1}(j)$, for $1\leq j\leq k$.
\end{notation}

Finally, we combine the description of the decorated trees in \cref{prop:lambdaCoef} with the forest formulas in \cref{thm:forestform0,thm:forestAntipode} in order to obtain the advertised formulas for the iterated reduced coproduct and the antipode in $\sS(\tT_+(V))$ in terms of Schröder trees.

\begin{theorem}
\label{thm:SymAntipode}
    The iterated reduced coproduct $\overline{\Delta}^{[k]}_\sS$ and the antipode $S_\sS$ in $\sS(\tT_+(V))$ can be written as
    \begin{align}
    \label{eq:iteratedCoproductS}
        \overline{\Delta}^{[k]}_\sS(w) =& \sum_{t\in \ST(n)} \sum_{f\in k\operatorname{-lin}(\widetilde{\sk}(t))} \tilde{c}(w,t,f),
\\ S_\sS(w) =& \sum_{t\in \ST(n)}(-1)^{i(t)} \tilde{w}_t \label{eq:SchroderSymAntipode}
    \end{align}
    for any word $w=a_1\cdots a_n\in \tT_+(V)$ and $k\geq1$.
\end{theorem}

\section{The antipode formula of $\tT(\tT_+(V))$ in terms of Schröder trees}
\label{sec:antipodeTT}
Motivated by the formulas obtained in \cref{thm:SymAntipode} for the antipode in $\sS(\tT_+(V))$, it is natural to ask if analogue formulas hold for the case of $\tT(\tT_+(V))$. In case of a positive answer, one should give a proper order in which words are multiplied due to the non-commutativity of the double tensor Hopf algebra. In this section, we study in detail this issue, and also present and prove combinatorial formulas for the iterated reduced coproducts and the antipode $S$ in $\tT(\tT_+(V))$ in terms of Schröder trees. More precisely, we obtain a Schröder tree-type formula for the action of the reduced coproduct on products of words, that will be used to obtain a combinatorial formula for the iterated reduced coproduct. Finally, by analyzing the cancellations in Takeuchi's formula \eqref{eq:takeuchi}, we will then arrive at a formula for the antipode in double tensor Hopf algebra.

\begin{remark}
    \label{rmk:naturaltotal}
    Let $t$ be a Schröder tree. Recall that the set of internal vertices of $t$, $\operatorname{Int}(t)$, i.e.~the set of vertices of $\sk(t)$, can be regarded as a poset. Furthermore, the planarity of $t$ induces a natural total order on $\operatorname{Int}(t)$. More precisely, if $v,w\in\operatorname{Int}(t)$, we have that $v<w$ if $w$ is a descendant of $v$, or if $v$ is to the left of $w$ in the planar representation of $t$. This total order is called the \emph{planar order on $\operatorname{Int}(t)$}.
\end{remark}

\subsection{Action of the reduced coproduct on products of words} We are interested in computing the antipode by using Takeuchi's formula
$$S(w) = \sum_{k\geq1} (-1)^k m^{[k]}\circ \overline{\Delta}^{[k]}(w)$$
for $w=a_1\cdots a_n\in \tT_+(V)$. By the definition of the iterated reduced coproduct, it will be convenient to have a combinatorial formula, in terms of Schröder trees, for the action of $\overline{\Delta}$ on products of words. To this end, recall that a {Schröder forest} is an ordered sequence of Schröder trees $F=(t_1,\ldots,t_m)$. Also, the skeleton of $F$, $\sk(F)$, can be regarded as a poset. In particular, it makes sense to consider $k$-linearizations of $\sk(F)$. For the particular case of $2$-linearizations, notice that a Schröder forest $F$ that has a 2-linearization $f$ is either a forest formed by two corollas, or a forest formed by corollas and at least a Schröder tree of height 2. Observe that the elements in $f^{-1}(1)$ form a subset of the roots of the Schröder trees in $F$.

\

Now, let $\mathfrak{w}=w_1|\cdots|w_m \in \tT_+(V)^{\otimes m}$ be a product of non-empty words with $|w_i| = n_i$ for any $1\leq i\leq m$. Also, let $F=(t_1,\ldots,t_m)\in \mathsf{FSch}(n_1,\ldots,n_m)$ be a Schröder forest, and consider a 2-linearization $f\in 2\operatorname{-lin}(\sk(F))$. Since $F$ is an ordered forest, we can order the elements in $f^{-1}(1) = \{v_1,\ldots,v_l\}$ from left to right. Now, for each $i\in[m]$, we set $B_i\subseteq [n_i]$ as follows: first, label consider the natural labelling of $t_i$ by the elements of $[n_i]$. Then, if the root of $t_i$ is not in $f^{-1}(1)$, then $B_i:= \emptyset$. Otherwise, if $v_{j_i}\in f^{-1}(1)$ is the root of $t_{i}$, then $B_{i}$ is given by the labels of the sectors adjacent to $v_{j_i}$. With this, we set
$$d_1(\mathfrak{w},F,f) := (w_1)_{B_1}|\cdots|(w_m)_{B_m}.$$

On the other hand, let $u_{i_1},\ldots,u_{i_{s_i}}$ be the internal vertices of $t_i$ such that $f(u_{i_j}) = 2$, for $1\leq i\leq m$. If the root if $t_i$ is in $f^{-1}(2)$, then $t_i$ is necessarily a corolla and $s_i=1$. In this case, we set $D_{i_1} := [n_i]$. Otherwise, we can order the vertices $u_{i_1}< \cdots < u_{i_{s_i}}$ according to the planar order of $t_i$. Next, by labelling the sectors of $t_i$ with its natural labelling, we define $D_{i_j}$ to be the subset of $[n_i]$ given by the labels of the sectors adjacent to $u_{i_j}$, for $1\leq j\leq s_i$. Hence, we can define
$$d_2(\mathfrak{w},F,f) :=  (w_1)_{D_{1_1}}|\cdots|(w_1)_{D_{1_{s_1}}}|(w_2)_{D_{2_1}} |\cdots|(w_2)_{D_{2_{s_2}}}|\cdots |(w_m)_{D_{m_1}}|\cdots|(w_m)_{D_{m_{s_m}}}.$$
From this notation, we can finally state a Schröder forest type formula for the reduced coproduct of products of words as follows.

\begin{lemma}
    \label{lem:CopProduct}
    Let $\mathfrak{w}=w_1|\cdots|w_m \in \tT_+(V)^{\otimes m}$ be a product of non-empty words in $\tT_+(V)$. If $n_i:= |w_i|$ for any $1\leq i\leq m$, then the action of the reduced coproduct on $\mathfrak{w}$ is given by
    \begin{equation}
    \label{eq:redCoponProducts}
        \overline{\Delta}(\mathfrak{w}) = \sum_{F\in \mathsf{FSch}(n_1,\ldots,n_m)} \sum_{h\in 2\operatorname{-lin}(\sk(F))} d_1(\mathfrak{w},F,h)\otimes d_2(\mathfrak{w}, F,h).
    \end{equation}
\end{lemma}
\begin{proof}
The proof is an adaptation of the proof of \cite[Lem. 11]{Menous2018}. For the reader's convenience, we present a complete proof. The proof is done by induction on $m\geq1$. The base case $m=1$ follows using the same argument that in \cref{lem:decoration} and taking into account the order of the product of subwords that appears in the second component of $\overline{\Delta}$. Now, assume that the relation in \eqref{eq:redCoponProducts} is true for a fixed $m\geq 1$. We will prove that the equation is also valid for $m+1$. 
\par For the proof, consider a product of non-empty words $\mathfrak{w} = \mathfrak{w}'|w_{m+1}$, where $\mathfrak{w}' = w_1|\cdots|w_m$. Since $\Delta$ is an algebra morphism, we have.
\begin{eqnarray}
 \nonumber   \overline{\Delta}(\mathfrak{w}) &=& \Delta(\mathfrak{w}'|w_{m+1}) - \mathfrak{w}\otimes\uno - \uno\otimes\mathfrak{w}
    \\ &=& \big(  \overline{\Delta}(\mathfrak{w}') + \mathfrak{w}'\otimes\uno + \uno\otimes \mathfrak{w}' \big)|\big( \overline{\Delta}(w_{m+1}) + w_{m+1}\otimes\uno + \uno \otimes w_{m+1} \big) - \mathfrak{w}\otimes\uno - \uno\otimes\mathfrak{w} \nonumber
    \\ &=& \overline{\Delta}(\mathfrak{w}')|\overline{\Delta}(w_{m+1}) + \overline{\Delta}(\mathfrak{w}')|(w_{m+1}\otimes \uno ) + \overline{\Delta}(\mathfrak{w}')|(\uno\otimes w_{m+1} ) \nonumber
    \\ & & + (\mathfrak{w}'\otimes\uno )| \overline{\Delta}(w_{m+1}) + (\uno\otimes \mathfrak{w}')| \overline{\Delta}(w_{m+1}) + \mathfrak{w}'\otimes w_{m+1} + w_{m+1}\otimes \mathfrak{w}'. \label{eq:auxLemaCopProdcutW}
\end{eqnarray}

Let us analyze the first of the seven terms on the right-hand side of the above equation. By the induction hypothesis, the term $\overline{\Delta}(\mathfrak{w}')|\overline{\Delta}(w_{m+1})$ is equal to
\begin{eqnarray*}
   \overline{\Delta}(\mathfrak{w}')|\overline{\Delta}(w_{m+1}) &=& \left( \sum_{F'\in \mathsf{FSch}(n_1,\ldots,n_m)} \sum_{h\in 2\operatorname{-lin}(\sk(F'))} d_1(\mathfrak{w}',F',h)\otimes d_2(\mathfrak{w}', F',h)
   \right)\\ & & \Big|\left(\sum_{t\in \mathsf{Sch}(n_{m+1})} \sum_{f\in 2\operatorname{-lin}(\sk(t))} d_1(w_{m+1},t,f)\otimes d_2(w_{m+1}, t,f)\right)
\end{eqnarray*}
First observe that, given $F'\in \mathsf{FSch}(n_1,\ldots,n_m)$ and $t\in\ST(n_{m+1})$, we can construct $F \in \mathsf{FSch}(n_1,\ldots,n_{m+1})$ by adding $t$ at the end of the sequence $F$. Furthermore, given $h\in 2\operatorname{-lin}(\sk(F'))$ and $f\in 2\operatorname{-lin}(\sk(t))$, we construct $g\in 2\operatorname{-lin}(\sk(F))$ by the setting $g(v) = h(v)$ if $v\in \operatorname{Vert}(\sk(F'))$ and $g(v) = f(v)$ if $v\in \operatorname{Vert}(\sk(t))$. Notice that, in particular,  $g$ is a $2$-linearization such that the sets $g^{-1}(i)\cap \operatorname{Vert}(\sk(F'))$ and $g^{-1}(i) \cap \operatorname{Vert}(\sk(t))$ are non-empty, for $i=1,2$. Hence, the above sum can be indexed by the described pair $(F,g)$, and the corresponding term in the sum is precisely 
\begin{equation}
\label{eq:AuxLemma4}
d_1(\mathfrak{w}',F',h)| d_1(w_{m+1},t,f)\otimes d_2(\mathfrak{w}', F',h)| d_2(w_{m+1}, t,f) =  d_1(\mathfrak{w},F,g)\otimes d_2(\mathfrak{w}, F,g).
\end{equation}

Working in the same way, it turns out that each of the remaining six terms in the right-hand side of the last equation in \eqref{eq:auxLemaCopProdcutW} can be rewritten as a double sum indexed by Schröder forests $F \in \mathsf{FSch}(n_1,\ldots,n_{m_1})$ given in the same way as in the first case, and a $2$-linearization $g$ of $\sk(F)$ given as follows:
\begin{itemize}
    \item for $ \overline{\Delta}(\mathfrak{w}')|(w_{m+1}\otimes \uno )$, $g$ is such that $g(v) = 1$ for any $v\in\operatorname{Vert}(\sk(t))$, and \linebreak $g^{-1}(i) \cap \operatorname{Vert}(\sk(F'))\neq\emptyset$ for $i=1,2$; 
    \item for $ \overline{\Delta}(\mathfrak{w}')|(\uno\otimes w_{m+1} )$, $g$ is such that $g(v) = 2$ for any $v\in\operatorname{Vert}(\sk(t))$, and\linebreak $g^{-1}(i) \cap \operatorname{Vert}(\sk(F'))\neq\emptyset$ for $i=1,2$; 
    \item for $ (\mathfrak{w}'\otimes\uno )| \overline{\Delta}(w_{m+1})$, $g$ is such that $g(v) = 1$ for any $v\in\operatorname{Vert}(\sk(F'))$, and \linebreak$g^{-1}(i) \cap \operatorname{Vert}(\sk(t))\neq\emptyset$ for $i=1,2$;  
    \item for $(\uno\otimes \mathfrak{w}')| \overline{\Delta}(w_{m+1}) $, $g$ is such that $g(v) = 2$ for any $v\in \operatorname{Vert}(\sk(F'))$, and \linebreak$g^{-1}(i) \cap \operatorname{Vert}(\sk(t))\neq\emptyset$ for $i=1,2$;  
    \item for $\mathfrak{w}'\otimes w_{m+1}$, $g$ is such that $g(v) =1$ for any $v\in \operatorname{Vert}(\sk(F'))$, and $g(u)=2$ for any $u\in \operatorname{Vert}(\sk(t))$; 
    \item for $w_{m+1}\otimes \mathfrak{w}'$, $g$ is such that $g(v) = 2$ for any $v\in \operatorname{Vert}(\sk(F'))$, and $g(u) = 1$ for any $u\in\operatorname{Vert}(\sk(t))$.
\end{itemize}
Conversely, it is clear that 
{\small{
$$\mathsf{FSch}(n_1,\ldots,n_m)\times \ST(n_{m+1}) \ni ((t_1,\ldots,t_m),t)\mapsto F = (t_1,\ldots,t_m,t)\in \mathsf{FSch}(n_1,\ldots,n_m,n_{m+1}) $$}}is a bijection, and any $g\in 2\operatorname{-lin}(\sk(F))$ can be obtained in a unique way as one of the previous seven cases discussed. Hence, we have that the right-hand side of \eqref{eq:auxLemaCopProdcutW} can be re-indexed as a sum of pairs $(F,g)\in \mathsf{FSch}(n_1,\ldots,n_m,n_{m+1})\times 2\operatorname{-lin}(\sk(F))$ and, by the definition of $d_i$, the corresponding term indexed by $(F,g)$ is as in the right-hand side of \eqref{eq:AuxLemma4}. Therefore, we finally conclude that
$$ 
        \overline{\Delta}(\mathfrak{w}) = \sum_{F\in \mathsf{FSch}(n_1,\ldots,n_{m+1})} \sum_{g\in 2\operatorname{-lin}(\sk(F))} d_1(\mathfrak{w},F,g)\otimes d_2(\mathfrak{w}, F,g),
$$
so that the induction and the proof are now complete.
\end{proof}

\subsection{Action of the iterated reduced coproduct} The formula obtained in \cref{eq:redCoponProducts} is one of the ingredients to obtain a Schröder tree-type formula for the iterated reduced coproduct. In the same way as the commutative case, the formula will be indexed by $k$-linearization. However, in order to take into account the non-commutativity in $\tT(\tT_+(V))$, it will be convenient to introduce the following notation.

\begin{notation}
\label{not:extension}
    Let $t\in\ST(n)$ and $f\in k\operatorname{-lin}(\sk(t))$. We denote $\overline{f}:\operatorname{sk}(t)\to [i(t)]$ the unique order-preserving bijection given by the following conditions:
    \begin{enumerate}
        \item if $f(v)<f(w)$ then $\overline{f}(v) < \overline{f}(w)$;
        \item if $f^{-1}(j) = \{v_1,\ldots,v_s\}$, where $v_1<\cdots <v_s$ according to the planar order, then $\overline{f}(v_i)= \overline{f}(v_1) + i-1$, for every $i=1,\ldots,s$.
    \end{enumerate}
\end{notation}

Let $t\in\ST(n)$ and consider its associated non-crossing partition $\pi(t)\in\NC(n)$. In addition, if $f$ is a $k$-linearization of $\sk(t)$, then the map $\overline{f}$ induces a total order on the blocks of $\pi(t)$ which respects the partial order given by the nestings of the blocks. In other words, the pair $(t,\overline{f})$ can be equivalently regarded as a monotone partition $\pi(t,\overline{f}) = (C_1,\ldots,C_{i(t)})$. By setting $j_1,\ldots,j_{k-1}$ to be the indexes such that $C_{j_{i-1}+1},C_{j_{i-1}+2},\ldots,C_{j_{i}}$ (with $j_0=0$ and $j_{k}=i(t)$) are the blocks whose associated internal vertices satisfy that $f(v) = i$, for $i=1,\ldots,k$, we can also write $\pi(t,\overline{f}) = (C_{j_0+1},\ldots,C_{j_k})$. Thus, for a word $w=a_1\cdots a_n\in V^{\otimes n}$, we define the element $c(w,t,f)\in \tT(\tT_+(V))^{\otimes k}$ by the recipe
    \begin{equation}
        c(w,t,f) := w_{C_{j_0+1}}|\cdots|w_{C_{j_1}}\otimes w_{C_{j_1+1}}|\cdots | w_{C_{j_2}}\otimes \cdots \otimes w_{C_{j_{k-1}+1}}|\cdots |w_{C_{j_k}}.
    \end{equation}
 
Observe that the element $c(w,t,f)$ is simply a $k$-order pure tensor obtained from the Schröder tree $t$, whose components are indexed by the $k$-linearization $f$ and, in the case that there are two or more elements with the same index, then we multiply them from left to right, following the definition of the coproduct $\Delta$. With this notation, we can now state and prove the following formula for the $k$-th iteration of the reduced coproduct in terms of Schröder trees and $k$-linearizations.

\begin{theorem}
\label{thm:iteratedDelta}
    For a word $w=a_1\cdots a_n \in \tT_+(V)$ and $k\geq1$, we have
    \begin{equation}
    \label{eq:iteratedDelta}
        \overline{\Delta}^{[k]}(w) = \sum_{t\in \ST(n)}\sum_{f\in k\operatorname{-lin}(\sk(t))} c(w,t,f).
    \end{equation}
\end{theorem}
\begin{proof}
    Analogously to the case of \cref{lem:CopProduct}, the present proof is an adaptation of the proof of \cite[Lem. 12]{Menous2018}, and we also provide a complete proof. The proof is done by induction on $k\geq2$. The base case $k=2$ follows from \cref{lem:decoration}. Now, assume that the relation in \eqref{eq:iteratedDelta} is true for a fixed $k\geq 2$. We will prove that the \eqref{eq:iteratedDelta} is valid for $k+1$. 
\par    Let $w=a_1\cdots a_n\in \tT_+(V)$. By the definition of the iterated reduced coproduct and the induction hypothesis, we have that
    \begin{eqnarray*}
        \overline{\Delta}^{[k+1]}(w) &=& (\id^{\otimes(k-1)}\otimes \overline{\Delta})\circ \overline{\Delta}^{[k]}(w)
        \\ &=& \sum_{t\in \ST(n)}\sum_{f\in k\operatorname{-lin}(\sk(t))} c_1^f\otimes\cdots \otimes c_{k-1}^f \otimes \overline{\Delta}(c_k^f),
    \end{eqnarray*}
    where we have written $c(w,t,f) = c_1^f\otimes \cdots \otimes  c_k^f$. If $c_k^f = w_1|\cdots |w_m$, where $n_i := |w_i|$ for $1\leq i\leq m$, we can use \cref{lem:CopProduct} in order to write 
    $$\overline{\Delta}(c_k^f)  = \sum_{F\in\mathsf{FSch}(n_1,\ldots,n_m)} \sum_{h\in2\operatorname{-lin}(\sk(F))} d_1^h \otimes d_2^h,$$
    so that we have
    $$\overline{\Delta}^{[k+1]}(w) = \sum_{\substack{t\in \ST(n)\\f\in k\operatorname{-lin}(\sk(t))}} \sum_{\substack{F \in \mathsf{FSch}(n_1,\ldots,n_m)\\h\in 2\operatorname{-lin}(\sk(F))}} c_1^f\otimes\cdots \otimes c_{k-1}^f \otimes d_1^h\otimes d_2^h.$$
    Our objective is to rearrange the above sum as a sum indexed by pairs $(t',g)$, where $t'\in\ST(n)$ and $g\in (k+1)\operatorname{-lin}(\sk(t'))$. First, consider $t\in \ST(n)$, $f\in k\operatorname{-lin}(\sk(t))$, $F\in \mathsf{FSch}(n_1,\ldots,n_m)$ and $h\in 2\operatorname{-lin}(\sk(F))$. By definition, a forest $F$ that provides a non-zero contribution in the above sum is either a forest formed by two corollas, or a forest formed by corollas and at least a Schröder tree of height 2. Every Schröder tree in $F$ appears by applying the coproduct on a factor of $c_k^f$, which is indexed by an internal vertex of $t$. Each of these internal vertices together with its adjacent sectors define a corolla. Thus, we can replace each of those corollas with the corresponding tree in $F$, obtaining in this way a Schröder tree $t'$ with $n+1$ leaves. On the other hand, observe that the vertices in the set $f^{-1}(k)$ form the subset of leaves of $\sk(t)$ that coincides with the roots of the trees in $F$. In particular, $h^{-1}(1) \subseteq f^{-1}(k)$. Thus, we can define $g\in (k+1)\operatorname{-lin}(\sk(t'))$ by setting $g(v) = m$ if $v\in f^{-1}(m)$ and $1\leq m <k$, $g(v) = k$ if $v\in h^{-1}(1)$, and $g(v) = k+1$ if $v\in h^{-1}(2)$.
    \par Conversely, let $t'\in \ST(n)$ and $g\in (k+1)\operatorname{-lin}(\sk(t'))$. Then, we can consider:
    \begin{itemize}
        \item $F$ to be the forest of Schröder trees given by the subtrees generated by the internal vertices $g^{-1}(k)\cup g^{-1}(k+1)$ as well as their children;
        \item $h:\sk(F)\to [2]$ to be the 2-linearization given by the standardization of the restriction of $g$ to $g^{-1}(k)\cup g^{-1}(k+1)$;
        \item $t$ to be the Schröder tree with $n+1$ leaves obtained from $t'$ and $F$ as follows: if $t_v$ is a Schröder tree in $F$ with $j_v$ leaves, we replace $t_v$ in $t'$ by a corolla with $j_v$ leaves;
        \item $f:\sk(t)\to [k]$ to be the $k$-linearization given as follows: $f(v) = g(v)$ if $g(v) < k$, and $f(v) = k$ if $g(v)\geq k$.
    \end{itemize}
    In particular, notice that $f$ is a $k$-linearization and $f^{-1}(k)$ are the leaves of $\sk(t)$ that coincide with the roots of the trees in $F$. Finally, observe that the above construction provides a bijection so that we can write
    \begin{eqnarray*}
    \overline{\Delta}^{[k+1]} &=& \sum_{t\in\ST(n)}\sum_{g\in (k+1)\operatorname{-lin}(t)} c_1^g\otimes \cdots c_k^g\otimes c_{k+1}^g
    \\ &=& \sum_{t\in\ST(n)}\sum_{g\in (k+1)\operatorname{-lin}(t)} c(w,t,g),
    \end{eqnarray*}
    as we wanted to prove.
\end{proof}

\subsection{Ascent-free linearizations} Our final step is to describe which $k$-linearizations will contribute to Takeuchi's formula for the antipode in $\tT(\tT_+(V))$. More precisely, starting from a word $w=a_1\cdots a_n \in \tT_+(V)$, we write
\begin{equation}
    w_{(t,f)} := m^{[k]}\circ c(w,t,f) = w_{C_1}|w_{C_2}|\cdots|w_{C_{i(t)}}, 
\end{equation}
for any $t\in \ST(n)$ and $f\in k\operatorname{-lin}(\sk(t))$. Here,  $\pi(t,\overline{f}) = (C_1,\ldots,C_{i(t)})$ is the monotone partition associated to the pair $(t,\overline{f})$ with $\overline{f}$ is as defined in \cref{not:extension}. It is readily to see that $w_{(t,f)} = w_{(t,\overline{f})}$, i.e.~$w_{(t,f)}$ only depends of $f$ via $\overline{f}$. We can also see that in Takeuchi's formula, we will have cancellations given by different $k$-linearizations $f$ that have the same $i(t)$-linearization $\overline{f}$. More precisely, we have
\begin{eqnarray}
 \nonumber   S(w) &=& \sum_{k\geq1} (-1)^k  \sum_{t\in \ST(n)}\sum_{f\in k\operatorname{-lin}(\sk(t))} w_{(t,f)}
    \\ &=& \sum_{t\in \ST(n)}\sum_{g\in i(t)\operatorname{-lin}(\sk(t))} w_{(t,g)}\left(\sum_{k\geq 1}\sum_{\substack{f\in k\operatorname{-lin}(\sk(t))\\ \overline{f}= g}} (-1)^k \right). \label{eq:auxascent}
\end{eqnarray}

In order to analyze which $i(t)$-linearizations will have a non-zero contribution to the above sum, we introduce the following notion.

\begin{definition}
    Let $t$ be a Schröder tree and  $(v_1,\ldots,v_{i(t)})$ be the sequence of internal vertices of $t$ ordered increasingly according to the planar order. Also, let $f\in k\operatorname{-lin}(\sk(t))$ as well as its associated $\overline{f}\in i(t)\operatorname{-lin}(\sk(t))$.  For $1\leq j< i(t)$, we say that $f$ has an \textit{ascent on j} if $v_j$ is not the parent of $v_{j+1}$ and $\overline{f}(v_j) < \overline{f}(v_{j+1})$. The quantity
    $$m(t,f) := |\{1\leq j< i(t)\,:\,f \mbox{ has an ascent on }j\}|$$
    is called the \textit{ascent number of t}.
\end{definition}

We notice that the definition of the ascent number  $m(t,f)$ only depends on $f$ via $\overline{f}$ so that $m(t,f) = m(t,\overline{f})$. In addition, for any Schröder tree $t$ we can construct a $i(t)$-linearization $f$ such that $m(t,f)=0$. Indeed, we can take the labelling induced by a mirrored pre-order traversal on $\sk(t)$: if $\sk(t) = B_+(t_1,\ldots,t_m)$, we visit the root of $\sk(t)$, then orderly visit $t_m,t_{m-1},\ldots,t_1$ by mirrored pre-order traversal. It is not difficult to show that there is no other $f$ such that $m(t,f)=0$, which motivates the following definition.

\begin{definition}
    Let $t$ be a Schröder tree. We say that $g\in i(t)\operatorname{-lin}(\sk(t))$ is \textit{ascent-free} if $m(t,g)=0$. The unique $f_t^{\mathrm{op}}\in i(t)\operatorname{-lin}(\sk(t))$ such that $m(t,f_t^{\mathrm{op}})=0$ is called the \textit{ascent-free linearization of $t$}.
\end{definition}

It turns out that the only $g\in i(t)\operatorname{-lin}(\sk(t))$ that have a non-zero contribution in \eqref{eq:auxascent} are those that are ascent-free. More precisely, we have the following lemma.

\begin{lemma}
    Let $t$ be a Schröder tree and consider $g\in i(t)\operatorname{-lin}(\sk(t))$. Then
    \begin{equation}
    \label{eq:lemmaCancel}
        \sum_{k\geq 1}\sum_{\substack{f\in k\operatorname{-lin}(\sk(t))\\ \overline{f}= g}} (-1)^k  = \left\{\begin{tabular}{cl}
            $(-1)^{i(t)}$ & if $g$ is ascent-free, \\[0.15cm]
            $0$ & otherwise.        \end{tabular}\right.
    \end{equation}
\end{lemma}
\begin{proof}
    We will prove the lemma by induction on $i(t)$, the number of internal vertices of $t$. The case $i(t)=1$ is trivial since the only $1$-linearization that exists is ascent-free. \par Now, assume that \eqref{eq:lemmaCancel} is valid for Schröder trees with at most $m$ internal vertices, and take a Schröder tree $t$ with $i(t) = m+1$. The idea is to describe the set of $k$-linearizations of $\sk(t)$ in terms of a smaller tree. More precisely, let $v_x$ be the internal vertex of $t$ such that $g(v_x) = i(t)$. Since $g$ is order-preserving, we have that $v_x$ is a leaf of $\sk(t)$, so that if $c$ stands for the corolla determined by $v_x$ and their children in $t$, we have that $t' = t\backslash c$ is a Schröder tree (replacing the root of $c$ by a leaf) with $i(t') = m$. In addition, $g$ restricted to $\operatorname{Int}(t')$, denoted by $g'$, is also a $i(t')$-linearization of $\sk(t')$. We then have the following two cases:
    \vspace{.1in}
    \par i) \textit{$g$ does not have an ascent on $i(t') = i(t)-1$}. In this case, for any $f$ $k$-linearization such that $\overline{f}=g$, we can equivalently find a $(k-1)$-linearization such that $\overline{f'} = g'$ by taking the restriction of $f$ to $\operatorname{Int}(t')=\operatorname{Vert}(\sk(t'))$. Conversely, given $f'\in (k-1)\operatorname{-lin}(\sk(t'))$ such that $\overline{f'}=g'$, we can find $f\in k\operatorname{-lin}(\sk(t))$ such that $\overline{f}=g$ by defining $f(v_x) = k$. Observe that we cannot define $f(v_x) = k-1$. Otherwise, if $v_y$ is the internal vertex of $t'$ such that $g'(v_y) = i(t')$, then $f'(v_y) = k-1$. Since $f'$ is strictly order-preserving, we have that $v_y$ is not the parent of $v_x$. In addition, since $g$ does not have an ascent on $i(t)-1$, then $v_x < v_y$ in the planar order on $\operatorname{Int}(t)$. However, $f(v_x) = f(v_y)$ would imply that $g(v_x)<g(v_y)$, which contradicts the fact that $v_x$ is maximal. Therefore, we have
    \begin{eqnarray*}
         \sum_{k\geq 1}\sum_{\substack{f\in k\operatorname{-lin}(\sk(t))\\ \overline{f}= g}} (-1)^k &=& -\sum_{k\geq 1}\sum_{\substack{f'\in k\operatorname{-lin}(\sk(t'))\\ \overline{f'}= g'}} (-1)^{k}  
         \\ &=& \left\{\begin{tabular}{cl}
            $-(-1)^{i(t')}$ & if $g'$ is ascent-free, \\[0.15cm]
            $-0$ & otherwise.        \end{tabular}\right.,
            \\ &=& \left\{\begin{tabular}{cl}
            $(-1)^{i(t)}$ & if $g$ is ascent-free, \\[0.15cm]
            $0$ & otherwise.        \end{tabular}\right.,
    \end{eqnarray*}
    where we used the induction hypothesis in the second equality, and in the last equality, we used that if $g$ does not have an ascent on $i(t)-1$, then $g$ is ascent-free if and only if $g'$ is ascent-free.
    \vspace{.1in}
    \par ii) \textit{g has an ascent on $i(t')$}. In this case, we have that $m(t,g)>0$ so that $g$ is not ascent-free. In a similar way to the previous case, an element $f'\in (k-1)\operatorname{-lin}(\sk(t'))$ produces the following two maps:
    \begin{enumerate}
        \item $f_1\in k\operatorname{-lin}(\sk(t))$ by setting $f_1(v_x) = k$ and $f_1(v) = f'(v)$ for any $v\in \operatorname{Vert}(\sk(t'))$; 
        \item $f_2\in (k-1)\operatorname{-lin}(\sk(t))$ by setting $f_2(v_x) = k-1$ and $f_2(v) = f'(v)$ for any $v\in \operatorname{Vert}(\sk(t'))$.
    \end{enumerate}
    Unlike the previous case, the fact that $g$ has an ascent on $i(t)-1$ implies that $f_2$ in case (2) above is well-defined. Conversely, any $f\in k\operatorname{-lin}(\sk(t))$ can be obtained in a unique way from $f' = f|_{\operatorname{Int}(t')}$ as $f_1$ or $f_2$ depending if $f(v_y) = f(v_x)$ or $f(v_y)<f(v_x)$ respectively, where $v_y$ is the internal vertex of $t'$ such that $g'(v_y) = i(t')$. Thus we have
    \begin{eqnarray*}
         \sum_{k\geq 1}\sum_{\substack{f\in k\operatorname{-lin}(\sk(t))\\ \overline{f}= g}} (-1)^k &=& \sum_{k\geq 1}\left( \sum_{\substack{f'\in (k-1)\operatorname{-lin}(\sk(t'))\\ \overline{f'}= g'}} (-1)^{k}  + \sum_{\substack{f'\in k\operatorname{-lin}(\sk(t'))\\ \overline{f'}= g'}} (-1)^{k} \right) 
         \\ &=& \left\{\begin{tabular}{cl}
            $-(-1)^{i(t')} + (-1)^{i(t')}$ & if $g'$ is ascent-free, \\[0.15cm]
            $-0$ & otherwise.        \end{tabular}\right.
            \\ &=& 0,
    \end{eqnarray*}
    where we used the induction hypothesis in the second equality. This completes the induction and therefore the lemma is proved.
\end{proof}

By combining \eqref{eq:auxascent} with the previous lemma, we finally arrive at one of the main results of the paper: a formula for the antipode in the double tensor Hopf algebra indexed by Schröder trees.

\begin{theorem}[Antipode formula in the double tensor Hopf algebra]
\label{thm:antipodetT}
    The action of antipode $S$ in $\tT(\tT_+(V))$ can be written as
    \begin{equation}
    \label{eq:antipodetT}
        S(w) = \sum_{t\in \ST(n)} (-1)^{i(t)}w_t,
    \end{equation}
    for any word $w= a_1\cdots a_n \in \tT_+(V)$, where $w_t: = w_{(t,f_t^{\mathrm{op}})}$ and $f_t^{\mathrm{op}}:\sk(t)\to[i(t)]$ is the ascent-free linearization of $t$.
\end{theorem}

As an immediate consequence of the above theorem, we obtain a formula for the inverse of a character on $\tT(\tT_+(V))$ in terms of Schröder trees.

\begin{corollary}
\label{cor:PhiInv}
    For a character $\Phi$ on $\tT(\tT_+(V))$, we have
    \begin{equation}
        \Phi^{*-1}(w) = (\Phi\circ S)(w) = \sum_{t\in \ST(n)}(-1)^{i(t)} \Phi_t(w),
    \end{equation}
    for any word $w=a_1\cdots a_n\in 
    \tT_+(V),$ where $\Phi_t(w):= \Phi(w_t) = \prod\limits_{B\in \pi(t)} \Phi(w_B) = \varphi_{\pi(t)}(a_1,\ldots,a_n).$
\end{corollary}
\begin{proof}
    The result follows immediately from the fact that $\Phi$ is multiplicative and \eqref{eq:antipodetT}.
\end{proof}

\section{Applications of the antipode formula in moment-cumulant relations}
\label{sec:applications}

In this final section, we apply the formula for the antipode in the double tensor Hopf algebra provided in \cref{thm:antipodetT} in the context of non-commutative probability theory. First, we recover the alternative presentation, given in \cite{JVNT,AC22}, of the formulae for the free, Boolean and monotone cumulants-to-moments formula in terms of prime, Boolean and general Schröder trees, respectively. We also present a new formula expressing the inverse of the moment map in the double tensor algebra as a signed sum of moments. Finally, we obtain a new formula for the free Wick map in terms of Schröder trees.

\subsection{Free cumulants in terms of Schröder trees}

Let $(\A,\varphi)$ be a non-commutative probability space and consider its free and Boolean cumulants $\{k_n\}_{n\geq1}$ and $\{b_n\}_{n\geq1}$, respectively. By Möbius inversion, the combinatorial moment-cumulant relations \eqref{eq:FreeMC} and \eqref{eq:BoolMC} can be inverted to write cumulants in terms of moments:
\begin{eqnarray}
 \label{eq:FreeCMt}   k_n(a_1,\ldots,a_n) &=& 
    \sum_{\pi\in \mathsf{NC}(n)} \operatorname{M\ddot{o}b}_{\mathsf{NC}_n}(\pi,1_n) \varphi_\pi(a_1,\ldots,a_n),\\
    b_n(a_1,\ldots,a_n) &=& 
    \sum_{\pi\in \mathsf{NCInt}(n)} (-1)^{|\pi|-1} \varphi_\pi(a_1,\ldots,a_n),\label{eq:BoolCMt}
\end{eqnarray}
for any $a_1,\ldots,a_n\in \A$. From the shuffle-algebraic point of view, the above formulas should be encompassed when inverting the respective shuffle fixed-point equations in \eqref{eq:fixedpointeqs}.

\

More precisely, consider the double tensor Hopf algebra $\tT(\tT_+(\A))$ as well as the character $\Phi$ on $\tT(\tT_+(\A))$ extending $\varphi$, as we considered in \eqref{eq:character}. From \cref{thm:link}, the infinitesimal character $\kappa$ given by $\Phi = \epsilon + \kappa\prec \Phi$ agrees with the free cumulant $k_n(a_1,\ldots,a_n)$ when we evaluate it on a word $w=a_1\cdots a_n\in \tT_+(\A).$ The previous equation can be written as $\Phi-\epsilon = \kappa\prec\Phi$. Using the shuffle identities \eqref{eq:shuffle2}, we have
    \begin{eqnarray*}
    \kappa &=& \kappa \prec \epsilon \\&=& \kappa\prec(\Phi*\Phi^{*-1})\\&=& (\kappa\prec\Phi)\prec \Phi^{*-1}\\ &=& (\Phi-\epsilon)\prec \Phi^{*-1}.
    \end{eqnarray*} 
If $S$ stands for the antipode in $\tT(\tT_+(\A))$, then \cref{cor:PhiInv} suggests the connection between Schröder trees and the cumulant-moments relations. More precisely, we can directly recover \cite[Thm. 4.2]{JVNT}, which was originally proved by means of non-commutative symmetric functions and an unshuffle bialgebra of decorated Schröder trees. 

\begin{proposition}[{\cite[Thm. 4.2]{JVNT}}]
    \label{thm:JVMT}
    Let $(\calA,\varphi)$ be a non-commutative probability space and $\{k_n\}_{n\geq1}$ be its free cumulants. Then, for any $a_1,\ldots,a_n \in \A$ we have:
    \begin{equation}
    \label{eq:jvmt}
        k_n(a_1,\ldots,a_n) = \sum_{t\in \mathsf{PSch}(n)} (-1)^{i(t)-1}\varphi_{\pi(t)}(a_1,\ldots,a_n).
    \end{equation}
\end{proposition}

The proof of the above proposition makes use of the following simple but useful bijection.

\begin{lemma}
\label{lem:AuxFree}
    For every $n \geq 0$, the set of prime Schröder trees $\PST(n)$ is in bijection with the set
    \begin{equation}
        \mathcal{S} :=\big\{(A,t_1,\ldots,t_{|K(A)|})\,:\,1\in A\subseteq[n],\; t_j\in \ST(|K_j|), \,\mbox{ for }1\leq j\leq |K(A)|\big\}.
    \end{equation}
\end{lemma}

\begin{proof}
First, we start with a prime Schröder tree $t\in \mathsf{PSch}(n)$ and consider its associated non-crossing partition $\pi(t) \in \mathsf{NC}(n)$. We take $A$ to be the block in $\pi(t)$ associated to the root of $t$. Since $t$ is prime, we have that $1\in A$. Moreover, if $K(A) = \{K_1,\ldots, K_s\}$ is the decomposition of $[n]\backslash A$ into its connected components, then for each $1\leq j\leq s$ we define $t_j\in \ST(|K_j|)$ to be the Schröder sub-tree of $t$ determined by the internal vertices of $t$ such that the union of their associated blocks in $\pi(t)$ is $K_j$, as well as all the descendants of such vertices. In this way, it is clear that $(A,t_1,\ldots,t_s)\in\mathcal{S}$. Conversely, take $(A,t_1,\ldots,t_s)\in\mathcal{S}$. For every $1\leq j\leq s$, we consider $\pi_{K_j}(t_j)$ to be the non-crossing partition in $\mathsf{NC}(K_j)$ obtained by applying the unique increasing bijection from $[|K_j|]$ to $K_j$ to  the blocks of $\pi(t_j)$. Then, we construct $t\in \PST(n)$ by grafting $t_1,\ldots,t_s$ to a corolla with $|A|+1$ leaves in such a way that $\pi(t) = \{A\}\sqcup \pi_{K_1}(t_1)\sqcup\cdots \sqcup \pi_{K_s}(t_s)$. Since $1\in A$, we have that $t$ is indeed prime. It is easy to see that both constructions are invertible to each other, so that we have a bijection between $\PST(n)$ and $\mathcal{S}$.
\end{proof}
\begin{proof}[Proof of \cref{thm:JVMT}]
    Consider the double tensor Hopf algebra $\tT(\tT_+(\A))$ as well as the character $\Phi$ on $\tT(\tT_+(\A))$ extending $\varphi$. We also consider $\kappa$ to be the infinitesimal character such that $\Phi = \mathcal{E}_\prec(\kappa)$, or equivalently, $\kappa = (\Phi-\epsilon)\prec \Phi^{*-1}$. Let $w=a_1\cdots a_n\in\A^{\otimes n}$. Using the definition of the left half-shuffle product and \cref{cor:PhiInv}, we have
    \begin{eqnarray*}
        \kappa(w) &=& \sum_{1 \in A\subseteq[n]} \Phi(w_A)(\Phi\circ S)(w_{K_1}|\cdots |w_{K_s})
        \\ &=& \sum_{1 \in A\subseteq[n]} \Phi(w_A)\prod_{j=1}^s \;\;\sum_{t_j\in \ST(|K_j|)} (-1)^{i(t_j)}\Phi_{t_i}(w_{K_i}),
    \end{eqnarray*}
Finally, by applying \cref{lem:AuxFree} to construct $t\in \PST(n)$ of a given $(A,t_1,\ldots,t_s)$ and noticing that $i(t) = 1 + i(t_1)+\cdots +i(t_s)$, we obtain
    $$\kappa(w) = \sum_{t\in \mathsf{PSch}(n)} (-1)^{i(t)-1} \Phi_t(w),$$
    where we can conclude by recalling that $ \Phi_t(w)=\varphi_{\pi(t)}(a_1,\ldots,a_n)$.
\end{proof}

\subsection{Boolean cumulants in terms of Schröder trees}

The authors of \cite{JVNT} proved that \eqref{eq:jvmt} implies the free cumulant-moment formula \eqref{eq:FreeCMt}, so that the expansion in terms of Schröder trees is finer than the expansion in terms of non-crossing partitions. The corresponding formula for the Boolean case was approached in \cite{AC22} by similar methods that in the free case. Unlike the free case, the expansion in terms of Schröder trees for the Boolean cumulant-moment formula is equivalent to the expansion in terms of interval partitions. For the sake of completeness, we present and prove the next result which exhibits the mentioned relation between Boolean cumulants and Schröder trees, analogously to \cref{thm:JVMT}. 

\begin{proposition}[{\cite[Prop. 6.3]{AC22}}]
\label{thm:AC22}
Let $(\calA,\varphi)$ be a non-commutative probability space and $\{b_n\}_{n\geq1}$ be its Boolean cumulants. Then, for any $a_1,\ldots,a_n\in\A$ we have:
\begin{equation}
    b_n(a_1,\ldots,a_n) = \sum_{t\in \mathsf{BSch}(n)} (-1)^{i(t)-1}\varphi_{\pi(t)}(a_1,\ldots,a_n).
\end{equation}
\end{proposition}

As for the previous section, we present a bijection which will be useful for the proof of the above proposition.

\begin{lemma}\label{lemmaforAC22}
For every $n \geq 0$, there is a bijection between
\begin{equation}
    \mathcal{S}_1:=\{(A,t): \emptyset \neq A \subseteq \{2, \hdots, n\}, \; \;t \in \mathsf{Sch}(|A|) \}
\end{equation}
and the set $\mathcal{S}_2$ of the pairs $(t',d)$ satisfying the following conditions:
\begin{enumerate}[i)]
    \item $t'\in \ST(n)\setminus\PST(n)$;
    \item the block $B\in\pi(t')$ such that $1\in B$ is the block associated to the leftmost leaf of $\sk(t')$;
    \item $d:\operatorname{Leaf}(\sk(t'))\to\{0,1\}$ is a map such that $d(v^\ell_{t'})=1$, where  $\operatorname{Leaf}(\sk(t))$ stands for the set of leaves of $\sk(t)$, and $v_{t'}^\ell$ stands for the leftmost leaf of $\sk(t)$.
\end{enumerate}
\end{lemma}

\begin{proof}
    Let $A\subseteq[n]$ be a non-empty set such that $1\not\in A$, and $t\in \ST(|A|)$. We can construct an element $t'\in\ST(n)$ as follows: let $K(A) = \{K_1,\ldots,K_s\}$ be the decomposition of $[n]\backslash A$ into its connected components. Then, for each $1\leq j\leq s$, let $c_j$ be a corolla with $|K_j|+1$ leaves. Also, let $\pi_A(t)$ be the non-crossing partition in $\mathsf{NC}(A)$ obtained by applying the unique increasing bijection from $[|A|]$ to $A$ to the blocks of $\pi(t)$. Next, construct $t'\in \ST(n)$ by grafting $c_1,\ldots,c_s$ to $t$ in such a way that $\pi(t') = \pi_A(t) \sqcup\{K_1,\ldots,K_s\}$. The fact that $1\not\in A$ implies that $1\in K_1$, so that $t'\in \ST(n)\backslash\PST(n)$, with $K_1$ being the block of $\pi(t')$ associated to the leftmost internal vertex of $t'$, denoted by $v^\ell_{t'}$. Furthermore, we can encode the corollas that are grafted to $t$ in order to obtain $t'$ by considering a colouring function $d:\operatorname{Leaf}(\sk(t'))\to \{0,1\}$. By identifying $\operatorname{Leaf}(\sk(t))$ as a subset of $\operatorname{Leaf}(\sk(t'))$, we set $d(v)=0$ if $v\in \operatorname{Leaf}(\sk(t))$ and $d(v)=1$ if $v\in \operatorname{Leaf}(\sk(t'))\backslash \operatorname{Leaf}(\sk(t)).$ In particular, we have that $d(v^\ell_{t'})=1$. 
    \par The above discussion shows that we have a map from $\mathcal{S}_1$ to $\mathcal{S}_2$. Even more, notice that the described procedure is reversible. This means, given $t'\in\ST(n)\backslash\mathsf{PSch}(n)$ and $d:\operatorname{Leaf}(\sk(t'))\to\{0,1\}$ with $d(v^\ell_{t'})=1$, and such that the block associated to $v^\ell_{t'}$ contains $1$, we can produce a subset $1\not\in A\subseteq [n]$ and a Schröder tree $t\in \ST(|A|)$ by deleting all the corollas given by the leaves $v$ of $\sk(t')$ such that $d(v)=1$, and setting $A$ to be the union of the blocks of $\pi(t')$ associated to the vertices in $\operatorname{Int}(t)\subset\operatorname{Int}(t')$. The facts that $t'$ is not a prime Schröder tree and $d(v^\ell_{t'})=1$ imply that the corolla given by $v^\ell_{t'}$ will be deleted so that $1\not\in A$. The above procedure produces a bijection between $\mathcal{S}_1$ and $\mathcal{S}_2$, as we wanted to show.
\end{proof}

\begin{proof}[Proof of \cref{thm:AC22}]
As in the proof of \cref{thm:JVMT}, let $\Phi$ be the character extending $\varphi$ on the double tensor Hopf algebra $\tT(\tT_+(\A))$. From \cref{thm:link}, we know that Boolean cumulants identify with the infinitesimal character $\beta$ which satisfies the right fixed-point equation $\Phi = \epsilon +  \Phi\succ\beta$. Equivalently, the shuffle identities \eqref{eq:shuffle2} imply that the previous fixed-point equation is equivalent to
\begin{equation*}
    \beta = (\Phi\circ S)\succ (\Phi-\epsilon) . 
\end{equation*} 
Let $w=a_1\cdots a_n\in\A^{\otimes n}$. Using the definition of the right half-shuffle product and \cref{cor:PhiInv}, we have
\begin{eqnarray*}
        \beta(w) &=& \sum_{1 \not\in A\subseteq[n]} (\Phi\circ S)(w_A)\Phi(w_{K_1}|\cdots |w_{K_s})
        \\ &=& \sum_{1 \not\in A\subseteq[n]} \left(\sum_{t\in \ST(|A|)} (-1)^{i(t)}\Phi_t(w_A) \right) \prod_{j=1}^s \Phi(w_{K_j}).
        \\ &=& \Phi(w) + \sum_{\substack{1 \not\in A\subseteq[n]\\A\neq\emptyset}} \left(\sum_{t\in \ST(|A|)} (-1)^{i(t)}\Phi_t(w_A) \right) \prod_{j=1}^s \Phi(w_{K_j}),
\end{eqnarray*}
where the first term in the last equality above follows from considering $A=\emptyset$ in the second equality. By applying the bijection provided by \cref{lemmaforAC22} to the indexing set of the double sum above, we have
    \begin{equation}
        \label{eq:auxBoolean}
    \beta(w) = \sum_{\substack{t\in\ST(n)\backslash \mathsf{PSch}(n)\\B_{v_{t}^\ell} \ni 1 }} \Phi_t(w) \left(\sum_{\substack{d:\operatorname{Leaf}(\sk(t))\to\{0,1\} \\d(v^\ell_{t})=1}} (-1)^{i(t) - |d^{-1}(1)|} \right).    
    \end{equation}
    where $v_{t}^\ell$ stands for the leftmost vertex of $\sk(t)$ and $B_{v_t^\ell}$ stands for the block in $\pi(t)$ associated to $v_t^\ell$. By rearranging according to the index $k=|d^{-1}(1)|$ and counting the number of coulorings $d$ that satisfy this condition, we obtain:
    \begin{eqnarray*}
    \sum_{\substack{d:\operatorname{Leaf}(\sk(t))\to\{0,1\} \\d(v^\ell_t)=1}} (-1)^{i(t) - |d^{-1}(1)|}  &=& (-1)^{i(t)-1} \sum_{k=0}^{|\operatorname{Leaf}(\sk(t))|-1} {|\operatorname{Leaf}(\sk(t))|-1\choose k}(-1)^{k} 
    \\ &=& (-1)^{i(t)-1} (1-1)^{|\operatorname{Leaf}(\sk(t))|-1} 
    \\[0.1cm] &=& \left\{\begin{array}{c c l}
    (-1)^{i(t)-1} &\quad &\mbox{if $|\operatorname{Leaf}(\sk(t))|=1$}, \\[0.1cm] 0& \quad&\mbox{otherwise.}
    \end{array}\right.
    \end{eqnarray*}
    Thus, the only Schröder trees $t\in \ST(n)\backslash\mathsf{PSch}(n)$ with $B_{v_t^\ell}\ni 1$ that contribute to the sum in \eqref{eq:auxBoolean} are those such that $|\operatorname{Leaf}(\sk(t))|=1$, and these trees are precisely the elements of $\mathsf{BSch}(n)$. Therefore, by recalling that $\Phi_t(w) = \varphi_{\pi(t)}(a_1,\ldots,a_n)$, we finally obtain
    $$\beta(w) = \sum_{t\in\mathsf{BSch}(n)} (-1)^{i(t)-1} \varphi_{\pi(t)}(a_1,\ldots,a_n),$$
    which completes the proof.
\end{proof}

\subsection{Monotone cumulants in terms of Schröder trees} In the work \cite{AC22}, the authors obtained the monotone cumulant-moment formula via the Hopf algebra of decorated Schröder trees described in \cite{JVNT} and the so-called \textit{Murua coefficients} (\cite{Murua}), which are defined by the expression
$$\omega(t) := \sum_{k=1}^{|t|}\frac{(-1)^{k-1}}{k}\omega_k(t),$$
for any rooted tree $t$, where $\omega_k(t)$ stands for the number of $k$-linearizations of $t$,  i.e.~\linebreak$\omega_k(t) = |k\operatorname{-lin}(t)|$. Besides in the work of Murua, these coefficients have also appeared in \cite[Rem. 12]{Murua} in the context of numerical analysis of PDEs, in \cite{calaque2011two} in the computation of the pre-Lie Magnus expansion in the free pre-Lie algebra on one generator, and recently in \cite{CEFPP22} in the framework of cumulant-cumulant relations in non-commutative probability.

\

We now present a proof of the monotone cumulant-moment formula which is not based on the calculation of the logarithm in the Hopf algebra of Schröder trees but on the Schröder trees-type coproduct formula used to calculate the antipode in $\tT(\tT_+(\A))$ obtained in \cref{thm:iteratedDelta}.

\begin{theorem}[{\cite[Thm. 1.1]{AC22}}]
     Let $(\calA,\varphi)$ be a non-commutative probability space and $\{h_n\}_{n\geq1}$ be the monotone cumulants. Then, for any $a_1,\ldots,a_n\in\A$ we have:
    \begin{equation}
        h_n(a_1,\ldots,a_n) = \sum_{t\in \ST(n)} \omega(\sk(t)) \varphi_{\pi(t)}(a_1,\ldots,a_n).
    \end{equation}
\end{theorem}
\begin{proof}
    By \cref{thm:mainEFP}, if $\Phi$ is the character extending $\varphi$ on the double tensor Hopf algebra $\tT(\tT_+(\A))$, then the evaluation of  $w=a_1\cdots a_n\in \calA^{\otimes n}$ on infinitesimal character $\rho = \log^*(\Phi)$ is the monotone cumulant $h_n(a_1\ldots,a_n).$ On the other hand, we have the expansion
    $$\rho = \log^*(\Phi) = \sum_{k\geq1}\frac{(-1)^{k-1}}{k}(\Phi-\epsilon)^{*k}.$$
    Let $w=a_1\cdots a_n\in \A^{\otimes n}$. Since $(\Phi-\epsilon)({\bf{1}}) = 0$, we can use the iterated reduced coproduct to compute its $k$-fold convolution power:
    $$(\Phi-\epsilon)^{*k}(w) = m^{[k]}_{\mathbb{C}}\circ \Phi^{\otimes k}\circ \overline{\Delta}^{[k]}(w),$$
    where $m^{[k]}_{\mathbb{C}}$ stands for the associative product of $k$ complex numbers. Hence by \cref{thm:iteratedDelta} and using that $\Phi$ is an algebra morphism, we have the following expression in terms of Schröder trees:
    \begin{eqnarray*}
        (\Phi-\epsilon)^{*k}(w) &=& \sum_{t\in \ST(n)}\sum_{f\in k\operatorname{-lin}(\sk(t))} m_\mathbb{C}^{[k]}\circ\Phi^{\otimes k}\big(c(w,t,f)\big)
        \\ &=& \sum_{t\in \ST(n)}\sum_{f\in k\operatorname{-lin}(\sk(t))} \Phi_t(w)
        \\ &=&\sum_{t\in \ST(n)}\omega_k(\sk(t)) \Phi_t(w),
    \end{eqnarray*}
    where we write $\Phi_t(w) = \prod_{B\in \pi(t)} \Phi(w_B)$. By replacing the above equation in the expansion for $\log^*(\Phi)$, we obtain:
    \begin{eqnarray*}
        \rho(w) &=& \sum_{k\geq1} \frac{(-1)^{k-1}}{k} \sum_{t\in \ST(n)}\omega_k(\sk(t)) \Phi_t(w)
        \\ &=& \sum_{t\in \ST(n)}
        \left(\sum_{k\geq1} \frac{(-1)^{k-1}}{k}\omega_k(\sk(t))\right)\varphi_{\pi(t)}(a_1,\ldots,a_n) 
        \\ &=& \sum_{t\in \ST(n)}\omega(\sk(t))\varphi_{\pi(t)}(a_1,\ldots,a_n), 
    \end{eqnarray*}
    as we wanted to prove.
\end{proof}

\subsection{A formula for $\Phi^{*-1}$ in terms of non-crossing partitions} Following the approach of the combinatorial formulas for non-commutative cumulants, the next result offers a formula for the evaluation of the inverse of a character on $\tT(\tT_+(\A))$ in terms of non-crossing partitions. 

\begin{proposition}
    Let $(\A,\varphi)$ be a non-commutative probability space, and consider $\Phi$ the character extending $\varphi$ on $\tT(\tT_+(\A))$. Then, for a word $w=a_1\cdots a_n\in \A^{\otimes n}$, we have that
    \begin{equation}
        \Phi^{*-1}(w) = \sum_{\pi\in\NC(n)} \operatorname{M\ddot{o}b}_{\NC(n+1)}(\hat\pi,1_{n+1}) \varphi_\pi(a_1,\ldots,a_n),
    \end{equation}
    where $\hat\pi$ is the non-crossing partition in $\NC(\{0,1,\ldots,n\})\cong\NC(n+1)$ given by $\hat{\pi}:= \big\{\{0\}\big\}\sqcup\pi$.
    \begin{proof}
        By \cref{cor:PhiInv}, we know that
        $$\Phi^{*-1}(w) = (\Phi\circ S)(w) = \sum_{t\in\ST(n)} (-1)^{i(t)} \varphi_{\pi(t)}(a_1,\ldots,a_n).$$
        Since $\varphi_{\pi(t)}$ depends only of the non-crossing partition $\pi(t)$ associated to $t$, we can group the above sum as
        \begin{eqnarray*}
        \Phi^{*-1}(w) &=& \sum_{\pi\in \NC(n)} \sum_{\substack{t\in\ST(n)\\ \pi(t) = \pi}} (-1)^{|\pi(t)|} \varphi_{\pi(t)}(a_1,\ldots,a_n)
        \\ &=& \sum_{\pi\in \NC(n)} |\{t\in\ST(n)\,:\,\pi(t) = \pi\}| (-1)^{|\pi|} \varphi_{\pi}(a_1,\ldots,a_n).
        \end{eqnarray*}
        Hence, given $\pi\in\NC(n)$, we need to count the number of $t\in\ST(n)$ such that $\pi(t) = \pi$. To this end, we recall that there is a two-to-one surjective map $P:\mathsf{PSch}(n+1)\to\ST(n)$ such that, for any $t\in\ST(n)$, $P^{-1}(t)$ contains exactly the following prime Schröder trees:        
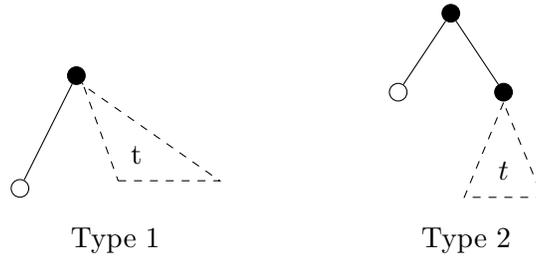
{
\begin{figure}[H]
    \centering
    \tikzset{
itria/.style={
  draw,dashed,shape border uses incircle,
  isosceles triangle,shape border rotate=90,yshift=-1.05cm},
rtria/.style={
  draw,dashed,shape border uses incircle,
  isosceles triangle,isosceles triangle apex angle=90,
  shape border rotate=-45,yshift=0.2cm,xshift=0.5cm},
ritria/.style={
  draw,dashed,shape border uses incircle,
  isosceles triangle,isosceles triangle apex angle=110,
  shape border rotate=-55,yshift=0.1cm},
letria/.style={
  draw,dashed,shape border uses incircle,
  isosceles triangle,isosceles triangle apex angle=110,
  shape border rotate=235,yshift=0.45cm, xshift = 0.05cm}
}
$$
\begin{array}{c c c}
\begin{tikzpicture}[sibling distance=1.5cm, level 2/.style={sibling distance =1.5cm}]
\node[circle,draw,fill=black, inner sep = 0.5ex] {}
    child{ node[circle, draw, inner sep = 0.5ex] {} }
    child{node[letria,draw] {t} };  
\end{tikzpicture} 
& \qquad \qquad \qquad & \begin{tikzpicture}[scale=0.7, sibling distance=2cm, level 2/.style={sibling distance =1.5cm}]
\node[circle,draw,fill=black, inner sep = 0.5ex] {}
    child{ node[circle, draw, inner sep = 0.5ex] {} }
    child{ node[circle, fill = black,draw, inner sep = 0.5ex] {}
                        { node[itria] {$t$} } 
                };
\end{tikzpicture} 
\\[0.2cm] \mbox{Type 1}& \qquad \qquad \qquad &\mbox{Type 2}
\end{array}
$$
\end{figure}
}
Now, consider $\hat{\pi}\in \NC(\{0,1,\ldots,n\})$ as above defined. Observe that the prime Schröder trees $\hat{t}$ such that $\pi(\hat{t})=\hat{\pi}$ can be only of Type 1, since otherwise, $\hat{t}$ being Type 2 would imply that $0$ belongs to a block of size larger than 1. Then $\hat{t}$ is the grafting of a single-vertex tree and $t\in \ST(n)$, and one can see that $\pi(t) = \pi$. In other words, we have a bijection between $\{t\in\ST(n) \,:\,\pi(t) = \pi\}$ and $\{s\in\mathsf{PSch}(n+1)\,:\,\pi(s) = \hat{\pi}\}$. By \cref{prop:NumberPST}, the latter set contains exactly $\operatorname{Abs}\big(\!\operatorname{M\ddot{o}b}_{\NC(n+1)}(\hat{\pi},1_{n+1})\big)$ 
    \begin{eqnarray*}
        \Phi^{*-1}(w) &=& \sum_{\pi\in \NC(n)}  (-1)^{|\hat{\pi}|-1}\operatorname{Abs}\big(\!\operatorname{M\ddot{o}b}_{\NC(n+1)}(\hat{\pi},1_{n+1})\big) \varphi_{\pi}(a_1,\ldots,a_n)
        \\ &=& \sum_{\pi\in \NC(n)}  \operatorname{M\ddot{o}b}_{\NC(n+1)}(\hat{\pi},1_{n+1}) \varphi_{\pi}(a_1,\ldots,a_n),
        \end{eqnarray*}
as we wanted to show. For a proof of the fact that $(-1)^{|\hat{\pi}|-1}$ is the suitable sign in order to obtain $\operatorname{M\ddot{o}b}_{\NC(n+1)}(\hat{\pi},1_{n+1})$ from its absolute value, the reader can check \cite[Sec. 5]{JVNT} for a relation between Schröder trees and the \textit{Kreweras complement} of a non-crossing partition (see \cite[Def. 9.21]{NSp}).
    \end{proof}
\end{proposition}

From the proof of the previous proposition, one can obtain the number of Schröder trees associated to the same non-crossing partition.

\begin{proposition}
    Let $n\geq1$. For any $\pi\in\mathsf{NC}(n)$, we have
    $$|\{t\in \ST(n)\;:\; \pi(t) = \pi\} | = \operatorname{Abs}\big(\!\operatorname{M\ddot{o}b}_{\mathsf{NC}(n+1)}(\hat{\pi},1_{n+1})\big),$$
    where $\hat\pi = \big\{\{0\}\big\}\sqcup \pi\in \NC(\{0,\ldots,n\})\cong\NC(n+1)$.
\end{proposition}

\subsection{Free Wick polynomials in terms of Schröder trees}

As a final application of the antipode formula in $\tT(\tT_+(\A))$, we present a formula that establishes the connection between the free Wick polynomials \eqref{eq:FreeWick} and Schröder trees. To the best of our knowledge, the formula presented in the following theorem represents a new and previously unexplored result.

\begin{theorem}[Free Wick Polynomials]

\label{thm:WickSch}

Let $(\A,\varphi)$ be a non-commutative probability space and consider the double tensor Hopf algebra $\tT(\tT_+(\A))$. If $\Phi$ is the character on $\tT(\tT_+(\A))$ extending $\varphi$, then the free Wick map $W$ \eqref{eq:Wick} of $\Phi$ satisfies, for any word $w=a_1\cdots a_n\in \A^{\otimes n}$:

\begin{eqnarray*}
   W(w) &=& \Phi^{*-1}(w)\uno + \sum_{t\in \ST(n)} (-1)^{i(t)-1}w_{B_{r}} \prod_{\substack{B\in \pi(t)\\B\neq B_r}} \Phi(w_B)
   \\ &=& \sum_{t\in \ST(n)} (-1)^{i(t)-1} \Big(w_{B_r} 
   - \Phi(w_{B_r})\uno \Big)\prod_{\substack{B\in \pi(t)\\B\neq B_r}} \Phi(w_B),
\end{eqnarray*}
where for each $t\in\ST(n)$, $B_r$ stands for the block in $\pi(t)$ associated to the root of $t$.
\end{theorem}

\begin{proof}
    By using the definition of the free Wick map $W$ in \cref{eq:Wick} and \cref{cor:PhiInv}, we have for any word $w=a_1\cdots a_n\in \A^{\otimes n}$:
    \begin{eqnarray*}
        W(w) &=& \big(\id\otimes \Phi^{*-1}\big)\circ \Delta (w)
        \\ &=& \sum_{A\subseteq [n]} w_A (\Phi\circ S)(w_{K_1}|\cdots|w_{K_s})
        \\ &=& \Phi^{*-1}(w)\uno +  \sum_{\emptyset \neq A\subseteq [n]} w_A \prod_{j=1}^s \sum_{t_j\in \ST(|K_j|)} (-1)^{i(t_j)} \Phi_{t_j}(w_{K_j}).
    \end{eqnarray*}
    Proceeding in a similar way as in the proof of \cref{lem:AuxFree} we have a bijection between $\ST(n)$ and the set
    $$\big\{(A,t_1,\ldots,t_{|K(A)|})\,:
    \, \emptyset\neq A\subseteq[n],\; t_j\in \ST(|K_j|)\,\mbox{ for }1\leq j\leq |K(A)|\big\},$$
    where, given $t\in \ST(n)$, $A$ is the block of $\pi(t)$ associated to the root of $t$, and for each $1\leq j\leq s$, $t_j\in \ST(|K_j|)$ is the Schröder sub-tree of $t$ determined by the internal vertices of $t$ such that the union of their associated blocks in $\pi(t)$ is $K_j$, as well as all the descendants of such vertices.
    \par  Next, given $t\in\ST(n)$ with corresponding tuple $(B_r,t_1,\ldots,t_s)$, we can use the fact that $\Phi$ is multiplicative in order to obtain
    $$\prod_{j=1}^s \Phi_{t_j}(w_{K_j}) = \prod_{j=1}^s \prod_{B\in\pi_{K_j}(t_j)} \Phi(w_{B}) = \prod_{\substack{B\in\pi(t)\\B\neq B_r}} \Phi(w_B),$$
    where $\pi_{K_j}(t_j)$ stands for the element in $\mathsf{NC}(K_j)$ obtained by applying the unique increasing bijection from $[|K_j|]$ to $K_j$ to every block of $\pi(t_j)$. Hence, we can write
    $$W(w) = \Phi^{*-1}(w)\uno + \sum_{t\in \ST(n)} (-1)^{i(t)-1} w_{B_r} \prod_{\substack{B\in\pi(t)\\ B\neq B_r}} \Phi(w_B),$$where we have also used that $i(t) = 1+i(t_1)+\cdots + i(t_s)$. 
    For the second equation, we use \cref{cor:PhiInv} to express $\Phi^{*-1}(w)$ as a sum in terms of Schröder trees and obtain:
    \begin{eqnarray*}
        W(w) &=&  \sum_{t\in \ST(n)} (-1)^{i(t)} \Phi_t(w)\uno + \sum_{t\in \ST(n)} (-1)^{i(t)-1} w_{B_r} \prod_{\substack{B\in\pi(t)\\ B\neq B_r}} \Phi(w_B)
        \\ &=& \sum_{t\in \ST(n)} (-1)^{i(t)-1}\left(w_{B_r} \prod_{\substack{B\in\pi(t)\\ B\neq B_r}} \Phi(w_B) - \Phi_t(w)\uno\right)
        \\ &=& \sum_{t\in \ST(n)} (-1)^{i(t)-1}\left(w_{B_r} \prod_{\substack{B\in\pi(t)\\ B\neq B_r}} \Phi(w_B) - \Phi\big(w_{B_r}\big)\uno \prod_{\substack{B\in\pi(t)\\ B\neq B_r}} \Phi(w_B)\right)
        \\ &=& \sum_{t\in \ST(n)} (-1)^{i(t)-1} \Big(w_{B_r} - \Phi\big(w_{B_r}\big)\uno\Big)\prod_{\substack{B\in\pi(t)\\ B\neq B_r}} \Phi(w_B),
    \end{eqnarray*}
where we have used that $\Phi$ is multiplicative in the third equality.
\end{proof}

\subsection*{Acknowledgments} Adrián Celestino is supported by the Austrian Science Fund (FWF) grant I 6232-N (WEAVE). Yannic Vargas is supported
by the Austrian Science Fund (FWF) grant I 5788.

\bibliographystyle{alpha}
\bibliography{schroeder}

\end{document}